\documentclass[12pt,a4paper,american]{article}
\usepackage[T1]{fontenc}
\usepackage[latin9]{inputenc}
\usepackage{amsmath}
\usepackage{amssymb}

\makeatletter

\newcommand{\lyxaddress}[1]{
\par {\raggedright #1
\vspace{1.4em}
\noindent\par}
}

\usepackage{amsthm}
\usepackage{mathrsfs}

\newtheorem{theorem}{Theorem}
\newtheorem{proposition}[theorem]{Proposition}
\newtheorem{lemma}[theorem]{Lemma}

\newtheorem*{corollary*}{Corollary}
\theoremstyle{remark}
\newtheorem{remark}[theorem]{Remark}

\newtheorem*{question*}{QUESTION}
\newtheorem*{remark*}{Remark}
\newcommand{\der}{\mathop\mathrm{der}\nolimits}
\newcommand{\spec}{\mathop\mathrm{spec}\nolimits}
\renewcommand{\Im}{\mathop\mathrm{Im}\nolimits}
\renewcommand{\Re}{\mathop\mathrm{Re}\nolimits}
\makeatother

\makeatother

\usepackage{babel}
\begin{document}

\title{Special functions and spectrum of Jacobi matrices}

\author{F.~\v{S}tampach$^{1}$, P.~\v{S}\v{t}ov\'\i\v{c}ek$^{2}$}

\date{{}}

\maketitle

\lyxaddress{$^{1}$Department of Applied Mathematics, Faculty of Information
Technology, Czech Technical University in~Prague, Kolejn\'\i~2,
160~00 Praha, Czech Republic}

\lyxaddress{$^{2}$Department of Mathematics, Faculty of Nuclear Science, Czech
Technical University in Prague, Trojanova 13, 12000 Praha, Czech Republic}
\begin{abstract}
\noindent Several examples of Jacobi matrices with an explicitly
solvable spectral problem are worked out in detail. In all discussed
cases the spectrum is discrete and coincides with the set of zeros
of a special function. Moreover, the components of corresponding eigenvectors
are expressible in terms of special functions as well. Our approach
is based on a recently developed formalism providing us with explicit
expressions for the characteristic function and eigenvectors of Jacobi
matrices. This is done under an assumption of a simple convergence
condition on matrix entries. Among the treated special functions there
are regular Coulomb wave functions, confluent hypergeometric functions,
q-Bessel functions and q-confluent hypergeometric functions. In addition,
in the case of q-Bessel functions, we derive several useful identities.
\end{abstract}
\vskip\baselineskip\noindent\emph{Keywords}: infinite Jacobi matrix,
spectral problem, special functions

\vskip0.5\baselineskip\noindent\emph{2010 Mathematical Subject
Classification}: 47B36, 39A70, 47A10, 33D15

\section{Introduction}

Special functions usually depend on a complex variable and an additional
parameter called order. Typically, they obey a three-term recurrence
relation with respect to the order. This is the basis of their relationship
to Jacobi (tridiagonal) matrices. In more detail, the zeros of an
appropriate special function are directly related to eigenvalues of
a Jacobi matrix operator, and components of corresponding eigenvectors
can be expressed in terms of special functions as well. One may also
say that the characteristic function of the (infinite) matrix operator
in question is written explicitly in terms of special functions. Particularly,
Gard and Zakraj\v{s}ek reported in \cite{GardZakrajsek} a matrix
equation approach for numerical computation of the zeros of Bessel
functions; on this point see also \cite{IkebeKikuchiFujishiro}. In
\cite{Ikebe}, Ikebe then showed that the same approach was applicable,
too, for determining the zeros of regular Coulomb wave functions.
In practical computations, an infinite tridiagonal matrix should be
truncated which raises a question of error estimates. Such an analysis
has been carried out in \cite{Ikebeatal93,MiyazakiKikuchiCaiIkebe}.

In \cite{StampachStovicek}, the authors initiated an approach to
a class of Jacobi matrices with discrete spectra. The basic tool is
a function $\mathfrak{F}$ depending on a countable number of variables.
In more detail, we define $\mathfrak{F}:D\rightarrow\mathbb{C}$,
\begin{equation}
\mathfrak{F}(x)=1+\sum_{m=1}^{\infty}(-1)^{m}\sum_{k_{1}=1}^{\infty}\,\sum_{k_{2}=k_{1}+2}^{\infty}\,\dots\,\sum_{k_{m}=k_{m-1}+2}^{\infty}\, x_{k_{1}}x_{k_{1}+1}x_{k_{2}}x_{k_{2}+1}\dots x_{k_{m}}x_{k_{m}+1},\label{eq:def_calF}
\end{equation}
where the set $D$ is formed by complex sequences $x=\{x_{k}\}_{k=1}^{\infty}$
obeying

\begin{equation}
\sum_{k=1}^{\infty}|x_{k}x_{k+1}|<\infty.\label{eq:x_converg_cond}
\end{equation}
For a finite number of variables we identify $\mathfrak{F}(x_{1},x_{2},\dots,x_{n})$
with $\mathfrak{F}(x)$ where $x=(x_{1},x_{2},\dots,x_{n},0,0,0,\dots)$.
By convention, we put $\mathfrak{F}(\emptyset)=1$ where $\emptyset$
is the empty sequence. Notice that the domain $D$ is not a linear
space though $\ell^{2}(\mathbb{N})\subset D$.

In the same paper, two examples are given of special functions expressed
directly in terms of $\mathfrak{F}$. The first example is concerned
with Bessel functions of the first kind. For $w,\nu\in\mathbb{C}$,
$\nu\notin-\mathbb{N}$, one has 
\begin{equation}
J_{\nu}(2w)=\frac{w^{\nu}}{\Gamma(\nu+1)}\,\mathfrak{F}\!\left(\left\{ \frac{w}{\nu+k}\right\} _{k=1}^{\infty}\right)\!.\label{eq:BesselJ_rel_F}
\end{equation}
Secondly, the formula 
\begin{equation}
\mathfrak{F}\!\left(\left\{ t^{k-1}w\right\} _{k=1}^{\infty}\right)=1+\sum_{m=1}^{\infty}(-1)^{m}\,\frac{t^{m(2m-1)}w^{2m}}{(1-t^{2})(1-t^{4})\dots(1-t^{2m})}={}_{0}\phi_{1}(;0;t^{2},-tw^{2})\label{eq:F_geom_tw}
\end{equation}
holds for $t,w\in\mathbb{C}$, $|t|<1$. Here $_{0}\phi_{1}$ is the
basic hypergeometric series (also called q-hypergeometric series)
being defined by
\[
_{0}\phi_{1}(;b;q,z)=\sum_{k=0}^{\infty}\frac{q^{k(k-1)}}{(q;q)_{k}(b;q)_{k}}\, z^{k},
\]
and
\[
(a;q)_{k}=\prod_{j=0}^{k-1}\left(1-aq^{j}\right),\mbox{ }k=0,1,2,\ldots,
\]
is the $q$-Pochhammer symbol, see \cite{GasperRahman}.

In \cite{StampachStovicek2}, the approach is further developed and
a construction in terms of $\mathfrak{F}$ of the characteristic function
of certain Jacobi matrices is established. As an application, a series
of examples of Jacobi matrices with explicitly expressible characteristic
functions is described. The method works well for Jacobi matrices
obeying a simple convergence condition imposed on the matrix entries
which is in principle dictated by condition (\ref{eq:x_converg_cond})
characterizing the domain of $\mathfrak{F}$.

In the current paper, we present more interesting examples of Jacobi
matrices whose spectrum coincides with the set of zeros of a particular
special function. As a byproduct, we provide examples of sequences
on which the function $\mathfrak{F}$ can be evaluated explicitly.
The paper is organized as follows. In Section~\ref{sec:Preliminaries}
we recall from \cite{StampachStovicek,StampachStovicek2} some basic
facts needed in the current paper. Section~\ref{sec:Coulomb-wave-functions}
is concerned with regular Coulomb wave functions. Here we reconsider
the example due to Ikebe while using our formalism. In Section~\ref{sec:Confluent-hypergeometric-fun}
we deal with confluent hypergeometric functions. Here we go beyond
the above mentioned convergence condition (see (\ref{eq:assum_sum_w})
below) which is violated in this example. Section~\ref{sec:Q-Bessel-functions}
is concerned with q-Bessel functions. This example is particular in
that respect that the constructed second order difference operator
is bilateral, i.e. it acts in $\ell^{2}(\mathbb{Z})$ rather than
in $\ell^{2}(\mathbb{N})$. We first derive several useful properties
of q-Bessel functions and then we use this knowledge to solve the
spectral problem for the bilateral difference operator fully explicitly.
Finally, another interplay between special functions, namely q-confluent
hypergeometric functions, and an appropriate Jacobi matrix is demonstrated
in Section~\ref{sec:Q-confluent-hypergeometric-fun}.

\section{Preliminaries \label{sec:Preliminaries}}

Let us recall from \cite{StampachStovicek,StampachStovicek2} some
basic facts concerning the function $\mathfrak{F}$ and its properties
and possible applications. First of all, quite crucial property of
$\mathfrak{F}$ is the recurrence rule 
\begin{equation}
\mathfrak{F}\!\left(\left\{ x_{k}\right\} _{k=1}^{\infty}\right)=\mathfrak{F}\!\left(\left\{ x_{k}\right\} _{k=2}^{\infty}\right)-x_{1}x_{2}\,\mathfrak{F}\!\left(\left\{ x_{k}\right\} _{k=3}^{\infty}\right).\label{eq:F_T_recur}
\end{equation}
In addition, $\mathfrak{F}(x_{1},x_{2},\dots,x_{k-1},x_{k})=\mathfrak{F}(x_{k},x_{k-1},\dots,x_{2},x_{1})$.
Furthermore, for $x\in D$,
\begin{equation}
\lim_{n\rightarrow\infty}\mathfrak{F}\!\left(\left\{ x_{k}\right\} _{k=n}^{\infty}\right)=1\ \ \text{and}\ \ \lim_{n\rightarrow\infty}\mathfrak{F}(x_{1},x_{2},\dots,x_{n})=\mathfrak{F}(x).\label{eq:lim_inf_F}
\end{equation}

Let us note that the definition of $\mathfrak{F}$ naturally extends
to more general ranges of indices. For any sequence $\left\{ x_{n}\right\} _{n=N_{1}}^{N_{2}}$,
$N_{1},N_{2}\in\mathbb{Z}\cup\{-\infty,+\infty\}$, $N_{1}\leq N_{2}+1$,
(if $N_{1}=N_{2}+1\in\mathbb{Z}$ then the sequence is considered
as empty) such that $\sum_{k=N_{1}}^{N_{2}-1}\left|x_{k}x_{k+1}\right|<\infty$
one defines
\[
\mathfrak{F}\!\left(\left\{ x_{k}\right\} _{k=N_{1}}^{N_{2}}\right)=1+\sum_{m=1}^{\infty}(-1)^{m}\sum_{k\in\mathcal{I}(N_{1},N_{2},m)}x_{k_{1}}x_{k_{1}+1}x_{k_{2}}x_{k_{2}+1}\ldots x_{k_{m}}x_{k_{m}+1}
\]
where
\[
\mathcal{I}(N_{1},N_{2},m)=\left\{ k\in\mathbb{Z}^{m};\, k_{j}+2\leq k_{j+1}\text{ }\text{for}\ 1\leq j\leq m-1,\ N_{1}\leq k_{1},\ k_{m}<N_{2}\right\} .
\]
With this definition, one has the generalized recurrence rule
\begin{equation}
\mathfrak{F}\!\left(\left\{ x_{k}\right\} _{k=N_{1}}^{N_{2}}\right)=\mathfrak{F}\!\left(\left\{ x_{k}\right\} _{k=N_{1}}^{n}\right)\mathfrak{F}\!\left(\left\{ x_{k}\right\} _{k=n+1}^{N_{2}}\right)-x_{n}x_{n+1}\,\mathfrak{F}\!\left(\left\{ x_{k}\right\} _{k=N_{1}}^{n-1}\right)\mathfrak{F}\!\left(\left\{ x_{k}\right\} _{k=n+2}^{N_{2}}\right)\label{eq:gen_recurrence}
\end{equation}
provided $n\in\mathbb{Z}$ satisfies $N_{1}\leq n<N_{2}$.

Let us denote by $J$ an infinite Jacobi matrix of the form
\begin{equation}
J=\begin{pmatrix}\lambda_{1} & w_{1}\\
w_{1} & \lambda_{2} & w_{2}\\
 & w_{2} & \lambda_{3} & w_{3}\\
 &  & \ddots & \ddots & \ddots
\end{pmatrix}\label{eq:matrixJ}
\end{equation}
where $\{w_{n};\ n\in\mathbb{N}\}\subset\mathbb{C\setminus}\{0\}$
and $\{\lambda_{n};\ n\in\mathbb{N}\}\subset\mathbb{C}$. In all examples
treated in the current paper, the matrix $J$ determines in a natural
way a unique closed operator in $\ell^{2}(\mathbb{N})$ (in other
words, $J_{\text{min}}=J_{\text{max}}$; see, for instance, \cite{Beckerman}).
If the matrix is real then the operator is self-adjoint. For the sake
of simplicity of the notation the operator is again denoted by $J.$
One notes, too, that all eigenvalues of $J$, if any, are simple since
any solution $\{x_{k}\}$ of the formal eigenvalue equation
\begin{equation}
\lambda_{1}x_{1}+w_{1}x_{2}=zx_{1},\ \ w_{k-1}x_{k-1}+\lambda_{k}x_{k}+w_{k}x_{k+1}=zx_{k}\ \ \text{\text{for }}k\geq2,\label{eq:eigenvalue}
\end{equation}
with $z\in\mathbb{C}$, is unambiguously determined by its first component
$x_{1}$.

Let $\{\gamma_{k}\}$ be any sequence fulfilling $\gamma_{k}\gamma_{k+1}=w_{k}$,
$k\in\mathbb{N}$. If $J_{n}$ is the principal $n\times n$ submatrix
of $J$ then
\begin{equation}
\det(J_{n}-zI_{n})=\left(\prod_{k=1}^{n}(\lambda_{k}-z)\right)\mathfrak{F}\!\left(\frac{\gamma_{1}^{\,2}}{\lambda_{1}-z},\frac{\gamma_{2}^{\,2}}{\lambda_{2}-z},\dots,\frac{\gamma_{n}^{\,2}}{\lambda_{n}-z}\right)\!.\label{eq:charpoly}
\end{equation}

The function $\mathfrak{F}$ can also be applied to bilateral difference
equations. Suppose that sequences $\left\{ w_{n}\right\} _{n=-\infty}^{\infty}$
and $\left\{ \zeta_{n}\right\} _{n=-\infty}^{\infty}$ are such that
$w_{n}\neq0$, $\zeta_{n}\neq0$ for all $n$, and

\[
\sum_{k=-\infty}^{\infty}\left|\frac{w_{k}^{\,2}}{\zeta_{k}\zeta_{k+1}}\right|<\infty.
\]
Consider the difference equation
\begin{equation}
w_{n}x_{n+1}-\zeta_{n}x_{n}+w_{n-1}x_{n-1}=0,\ n\in\mathbb{Z}.\label{eq:diff_eq_inf}
\end{equation}
Define the sequence $\left\{ \mathcal{P}_{n}\right\} _{n\in\mathbb{Z}}$
by $\mathcal{P}_{0}=1$ and $\mathcal{P}_{n+1}=(w_{n}/\zeta_{n+1})\mathcal{P}_{n}$
for all $n$. The sequence $\{\gamma_{n}\}_{n\in\mathbb{Z}}$ is again
defined by the rule $\gamma_{n}\gamma_{n+1}=w_{n}$ for all $n\in\mathbb{Z}$,
and any choice of $\gamma_{1}\neq0$. Then the sequences $\left\{ f_{n}\right\} _{n\in\mathbb{Z}}$
and $\left\{ g_{n}\right\} _{n\in\mathbb{Z}}$,

\begin{equation}
f_{n}=\mathcal{P}_{n}\,\mathfrak{F}\!\left(\left\{ \frac{\gamma_{k}^{\,2}}{\zeta_{k}}\right\} _{k=n+1}^{\infty}\right)\!,\ g_{n}=\frac{1}{w_{n-1}\mathcal{P}_{n-1}}\,\mathfrak{F}\!\left(\left\{ \frac{\gamma_{k}^{\,2}}{\zeta_{k}}\right\} _{k=-\infty}^{n-1}\right)\!,\label{eq:sols_diff_eq_inf}
\end{equation}
represent two solutions of the bilateral difference equation (\ref{eq:diff_eq_inf}).
With the usual definition of the Wronskian, $\mathcal{W}(f,g)=w_{n}\left(f_{n}g_{n+1}-f_{n+1}g_{n}\right)$,
one has
\begin{equation}
\mathcal{W}(f,g)=\,\mathfrak{F}\!\left(\left\{ \gamma_{n}^{\,2}\big/\zeta_{n}\right\} _{n=-\infty}^{\infty}\right)\!.\label{eq:Wronskian_fg}
\end{equation}

For $\lambda=\{\lambda_{n}\}_{n=1}^{\infty}$ let us denote $\mathbb{C}_{0}^{\lambda}:=\mathbb{C}\setminus\overline{\{\lambda_{n};\, n\in\mathbb{N}\}}$,
and let $\der(\lambda)$ stand for the set of all finite accumulation
points of the sequence $\lambda$. Further, for $z\in\mathbb{C}\setminus\der(\lambda)$,
let $r(z)$ be the number of members of the sequence $\lambda$ coinciding
with $z$ (hence $r(z)=0$ for $z\in\mathbb{C}_{0}^{\lambda}$). We
assume everywhere that $\mathbb{C}_{0}^{\lambda}\neq\emptyset$.

Suppose
\begin{equation}
\sum_{n=1}^{\infty}\left|\frac{w_{n}^{\,2}}{(\lambda_{n}-z_{0})(\lambda_{n+1}-z_{0})}\right|<\infty\label{eq:assum_sum_w}
\end{equation}
for at least one $z_{0}\in\mathbb{C}_{0}^{\lambda}$. Then (\ref{eq:assum_sum_w})
is true for all $z_{0}\in\mathbb{C}_{0}^{\lambda}$ \cite{StampachStovicek2}.
In particular, the following definitions make good sense. For $k\in\mathbb{Z}_{+}$
($\mathbb{Z}_{+}$ standing for nonnegative integers) and $z\in\mathbb{C}\setminus\der(\lambda)$
put
\begin{equation}
\xi_{k}(z):=\lim_{u\to z}\,(u-z)^{r(z)}\left(\prod_{l=1}^{k}\,\frac{w_{l-1}}{u-\lambda_{l}}\right)\!\mathfrak{F}\!\left(\left\{ \frac{\gamma_{l}^{\,2}}{\lambda_{l}-u}\right\} _{l=k+1}^{\infty}\right)\!.\label{eq:def_xi_k}
\end{equation}
Here one sets $w_{0}:=1$. Particularly, for $z\in\mathbb{C}_{0}^{\lambda}$,
one simply has
\begin{equation}
\xi_{k}(z)=\left(\prod_{l=1}^{k}\,\frac{w_{l-1}}{z-\lambda_{l}}\right)\mathfrak{F}\!\left(\left\{ \frac{\gamma_{l}^{\,2}}{\lambda_{l}-z}\right\} _{l=k+1}^{\infty}\right)\label{eq:def_xi_k_meromorph}
\end{equation}
(this is in fact nothing but the solution $f_{n}$ from (\ref{eq:sols_diff_eq_inf})
restricted to nonnegative indices). All functions $\xi_{k}(z)$, $k\in\mathbb{Z}_{+}$,
are holomorphic on $\mathbb{C}_{0}^{\lambda}$ and extend to meromorphic
functions on $\mathbb{C}\setminus\der(\lambda)$, with poles at the
points $z=\lambda_{n}$, $n\in\mathbb{N}$, and with orders of the
poles not exceeding $r(z)$. This justifies definition (\ref{eq:def_xi_k}).

The sequence $\{\xi_{k}(z)\}$ solves the second-order difference
equation
\begin{equation}
w_{k-1}x_{k-1}+(\lambda_{k}-z)x_{k}+w_{k}x_{k+1}=0\ \ \text{\text{for }}k\geq2.\label{eq:2ndorder_diff_eq_param_z}
\end{equation}
In addition, $(\lambda_{1}-z)\xi_{1}(z)+w_{1}\xi_{2}(z)=0$ provided
$\xi_{0}(z)=0$. Proceeding this way one can show \cite[Section~3.3]{StampachStovicek2}
that if $\xi_{0}(z)$ does not vanish identically on $\mathbb{C}_{0}^{\lambda}$
then
\begin{equation}
\spec(J)\setminus\der(\lambda)=\{z\in\mathbb{C}\setminus\der(\lambda);\ \xi_{0}(z)=0\}.\label{eq:specJ_xi0}
\end{equation}
Moreover, if $z\in\mathbb{C}\setminus\der(\lambda)$ is an eigenvalue
of $J$ then $\xi(z):=\left(\xi_{1}(z),\xi_{2}(z),\xi_{3}(z),\ldots\right)$
is a corresponding eigenvector. If $J$ is real and $z\in\mathbb{R}\cap\mathbb{C}_{0}^{\lambda}$
is an eigenvalue then $\|\xi(z)\|^{2}=\xi'_{0}(z)\xi_{1}(z)$. Finally,
let us remark that the Weyl m-function can be expressed as $m(z)=\xi_{1}(z)/\xi_{0}(z)$.

\begin{lemma} \label{lem:identity_FF} For $p,r,\ell\in\mathbb{N}$,
$1<p\leq r+1\leq\ell$, and any $\ell$-tuple of complex numbers $x_{j}$,
$1\leq j\leq\ell$, it holds true that
\begin{eqnarray*}
 &  & \mathfrak{F}\!\left(\left\{ x_{j}\right\} _{j=1}^{r}\right)\mathfrak{F}\!\left(\left\{ x_{j}\right\} _{j=p}^{\ell}\right)-\mathfrak{F}\!\left(\left\{ x_{j}\right\} _{j=1}^{\ell}\right)\mathfrak{F}\!\left(\left\{ x_{j}\right\} _{j=p}^{r}\right)\\
 &  & =\left(\prod_{j=p-1}^{r}x_{j}x_{j+1}\right)\mathfrak{F}\!\left(\left\{ x_{j}\right\} _{j=1}^{p-2}\right)\mathfrak{F}\!\left(\left\{ x_{j}\right\} _{j=r+2}^{\ell}\right)
\end{eqnarray*}
If  $p,r\in\mathbb{N}$, $1<p\leq r+1$, and a complex sequence $\left\{ x_{j}\right\} _{j=1}^{\infty}$
fulfills (\ref{eq:x_converg_cond}) then
\begin{eqnarray*}
 &  & \mathfrak{F}\!\left(\left\{ x_{j}\right\} _{j=1}^{r}\right)\mathfrak{F}\!\left(\left\{ x_{j}\right\} _{j=p}^{\infty}\right)-\mathfrak{F}\!\left(\left\{ x_{j}\right\} _{j=p}^{r}\right)\mathfrak{F}\!\left(\left\{ x_{j}\right\} _{j=1}^{\infty}\right)\\
 &  & =\left(\prod_{j=p-1}^{r}x_{j}x_{j+1}\right)\mathfrak{F}\!\left(\left\{ x_{j}\right\} _{j=1}^{p-2}\right)\mathfrak{F}\!\left(\left\{ x_{j}\right\} _{j=r+2}^{\infty}\right)\!.
\end{eqnarray*}
\end{lemma}

\begin{proof} Suppose $\left\{ z_{j}\right\} _{j=-\infty}^{\infty}$
is any nonvanishing bilateral complex sequence. In \cite[Section~2]{StampachStovicek2}
it is shown (under somewhat more general circumstances) that there
exists an antisymmetric matrix $\mathfrak{J}(m,n)$, $m,n\in\mathbb{Z}$,
such that
\[
\mathfrak{J}(m,n)=\left(\prod_{j=m+1}^{n-1}\frac{1}{z_{j}}\right)\mathfrak{F}\!\left(z_{m+1},z_{m+2},\ldots,z_{n-1}\right)
\]
for $m<n$, and $\mathfrak{J}(m,k)\mathfrak{J}(n,\ell)-\mathfrak{J}(m,\ell)\mathfrak{J}(n,k)=\mathfrak{J}(m,n)\mathfrak{J}(k,\ell)$
for all $m,n,k,\ell\in\mathbb{Z}$. In particular, assuming that indices
$p$, $r$, $\ell$ obey the restrictions from the lemma, 
\[
\mathfrak{J}(0,r+1)\mathfrak{J}(p-1,\ell+1)-\mathfrak{J}(0,\ell+1)\mathfrak{J}(p-1,r+1)=\mathfrak{J}(0,p-1)\mathfrak{J}(r+1,\ell+1).
\]
After obvious cancellations in this equation one can drop the assumption
on nonvanishing sequences. The lemma readily follows. \end{proof}

\begin{lemma} \label{lem:Fn_eq_calF} Let $x=\{x_{n}\}_{n=1}^{\infty}$
be a nonvanishing complex sequence satisfying (\ref{eq:x_converg_cond}).
Then
\begin{equation}
F_{n}:=\mathfrak{F}\!\left(\{x_{k}\}_{k=n}^{\infty}\right)\!,\ n\in\mathbb{N},\label{eq:Fn_def}
\end{equation}
is the unique solution of the second order difference equation
\begin{equation}
F_{n}-F_{n+1}+x_{n}x_{n+1}F_{n+2}=0,\ n\in\mathbb{N},\label{eq:F_diff_eq}
\end{equation}
satisfying the boundary condition $\lim_{n\rightarrow\infty}F_{n}=1$.
\end{lemma}

\begin{proof} The sequence $\{F_{n}\}$ defined in (\ref{eq:Fn_def})
fulfills all requirements, as stated in (\ref{eq:F_T_recur}) and
(\ref{eq:lim_inf_F}). It suffices to show that there exists another
solution $\{G_{n}\}$ of (\ref{eq:F_diff_eq}) such that $\lim_{n\rightarrow\infty}G_{n}=\infty$.
If $F_{1}=\mathfrak{F}(x)\neq0$ then $\{G_{n}\}$ can be defined
by $G_{1}=0$ and
\begin{equation}
G_{n}=\left(\prod_{k=1}^{n-2}\frac{1}{x_{k}x_{k+1}}\right)\mathfrak{F}\!\left(\{x_{k}\}_{k=1}^{n-2}\right),\ \ \text{for}\ n\geq2.\label{eq:Gn_def}
\end{equation}
If $F_{1}=0$ then necessarily $F_{2}\neq0$ since otherwise (\ref{eq:F_diff_eq})
would imply $F_{n}=0$ for all $n$ which is impossible. Hence in
that case one can shift the index by $1$, i.e. one can put $G_{2}=0$,
\[
G_{n}=\left(\prod_{k=2}^{n-2}\frac{1}{x_{k}x_{k+1}}\right)\mathfrak{F}\!\left(\{x_{k}\}_{k=2}^{n-2}\right),\ \ \text{for}\ n\geq3,
\]
(and $G_{1}=-x_{1}x_{2}$). In any case, $F_{n}$ is the minimal solution
of (\ref{eq:F_diff_eq}), see \cite{Gautschi}. \end{proof}

\begin{remark} If $\mathfrak{F}(x)=0$ then $\mathfrak{F}(x_{1},x_{2},\dots,x_{n})$
tends to $0$ as $n\to\infty$ quite rapidly, more precisely,
\begin{equation}
\mathfrak{F}(x_{1},x_{2},\dots,x_{n+1})=o\!\left(\prod_{k=1}^{n}x_{k}x_{k+1}\right)\!,\ \ \text{as}\ n\rightarrow\infty.\label{eq:calF_decay}
\end{equation}
In fact, if $\mathfrak{F}(x)=0$ then $F_{2}\neq0$ and the solutions
$\{F_{n}\}$ and $\{G_{n}\}$ defined in (\ref{eq:Fn_def}) and (\ref{eq:Gn_def}),
respectively, are linearly dependent, $F_{n}=F_{2}\, G_{n}$, $\forall n$.
Sending $n$ to infinity one gets
\[
1=F_{2}\,\lim_{n\rightarrow\infty}\left(\prod_{k=1}^{n}\frac{1}{x_{k}x_{k+1}}\right)\!\mathfrak{F}\!\left(\{x_{k}\}_{k=1}^{n}\right)=\lim_{n\rightarrow\infty}\left(\prod_{k=1}^{n}\frac{1}{x_{k}x_{k+1}}\right)\!\mathfrak{F}\!\left(\{x_{k}\}_{k=1}^{n}\right)\mathfrak{F}\!\left(\{x_{k}\}_{k=2}^{n+1}\right).
\]
Now, Lemma~\ref{lem:identity_FF} provides us with the identity
\[
\mathfrak{F}\!\left(\{x_{k}\}_{k=1}^{n}\right)\mathfrak{F}\!\left(\{x_{k}\}_{k=2}^{n+1}\right)-\mathfrak{F}\!\left(\{x_{k}\}_{k=1}^{n+1}\right)\mathfrak{F}\!\left(\{x_{k}\}_{k=2}^{n}\right)=\prod_{k=1}^{n}x_{k}x_{k+1},
\]
and so one arrives at the equation
\[
\lim_{n\rightarrow\infty}\left(\prod_{k=1}^{n}\frac{1}{x_{k}x_{k+1}}\right)\mathfrak{F}\!\left(\{x_{k}\}_{k=1}^{n+1}\right)\mathfrak{F}\!\left(\{x_{k}\}_{k=2}^{n}\right)=0.
\]
Since $\mathfrak{F}\!\left(\{x_{k}\}_{k=2}^{\infty}\right)\neq0$
this shows (\ref{eq:calF_decay}). \end{remark}

\section{Coulomb wave functions\label{sec:Coulomb-wave-functions}}

For $x>1$, $y\in\mathbb{R}$, put
\[
\lambda(x,y)=\frac{y}{(x-1)x},\ w(x,y)=\frac{1}{x}\,\sqrt{\frac{x^{2}+y^{2}}{4x^{2}-1}}\,,
\]
and
\[
\gamma(x,y)=\frac{\Gamma\!\left(\frac{1}{2}x\right)}{\sqrt{2x-1}\,\Gamma\!\left(\frac{1}{2}(x+1)\right)}\left|\frac{\Gamma\!\left(\frac{1}{2}(x+iy+1)\right)}{\Gamma\!\left(\frac{1}{2}(x+iy)\right)}\right|\!.
\]
Then $\gamma(x,y)\gamma(x+1,y)=w(x,y)$. For $\mu>0$, $\nu\in\mathbb{R}$,
consider the Jacobi matrix $J=J(\mu,\nu)$ of the form (\ref{eq:matrixJ}),
with
\begin{equation}
\lambda_{k}=\lambda(\mu+k,\nu),\ w_{k}=w(\mu+k,\nu),\ k=1,2,3,\ldots.\label{eq:coulomb_lbd_w}
\end{equation}
Similarly, $\gamma_{k}=\gamma(\mu+k,\nu)$. Clearly, the matrix $J(\mu,\nu)$
represents a Hermitian Hilbert-Schmidt operator in $\ell^{2}(\mathbb{N})$.
Moreover, the convergence condition (\ref{eq:assum_sum_w}) is satisfied
for any $z_{0}\in\mathbb{C}\backslash\{0\}$ such that $z_{0}\neq\lambda_{k}$,
$\forall k\in\mathbb{N}$.

Recall the definition of regular Coulomb wave functions \cite[Eq.~14.1.3]{AbramowitzStegun}
\begin{equation}
F_{L}(\eta,\rho)=2^{L}e^{-\pi\eta/2}\,\frac{|\Gamma(L+1+i\eta)|}{\Gamma(2L+2)}\,\rho^{L+1}e^{-i\rho}\,_{1}F_{1}(L+1-i\eta;2L+2;2i\rho),\label{eq:def_Coulomb_wave}
\end{equation}
valid for $L\in\mathbb{Z}_{+}$, $\eta\in\mathbb{R}$, $\rho>0$.
Let us remark that, though not obvious from its form, the values of
the regular Coulomb wave function in the indicated range are real.
But nothing prevents us to extend, by analyticity, the Coulomb wave
function to the values $L>-1$ and $\rho\in\mathbb{C}$ (assuming
that a proper branch of $\rho^{L+1}$ has been chosen).

As observed in \cite{Ikebe}, the eigenvalue equation for $J(\mu,\nu)$
may be written in the form $F_{\mu-1}(-\nu,z^{-1})=0$. Moreover,
if $z\neq0$ is an eigenvalue of $J(\mu,\nu)$ then the components
$v_{n}(z)$, $n\in\mathbb{N}$, of a corresponding eigenvector $v(z)$
are proportional to $\sqrt{2\mu+2n-1}\, F_{\mu+n-1}(-\nu,z^{-1})$.
Thus, using definition (\ref{eq:def_Coulomb_wave}), one can write

\begin{equation}
\text{spec}(J(\mu,\nu))\backslash\{0\}=\left\{ \zeta^{-1};\, e^{-i\zeta}\,_{1}F_{1}(\mu+i\nu;2\mu;2i\zeta)=0\right\} \label{eq:specJ_Cwave}
\end{equation}
and
\begin{equation}
v_{n}(\zeta^{-1})=\sqrt{2\mu+2n-1}\,\frac{|\Gamma(\mu+n+i\nu)|}{\Gamma(2\mu+2n)}\,(2\zeta)^{n-1}\, e^{-i\zeta}\,_{1}F_{1}(\mu+n+i\nu;2\mu+2n;2i\zeta).\label{eq:eigenvJ_Cwave}
\end{equation}

Here we wish to shortly reconsider this example while using our formalism.

\begin{proposition} Under the above assumptions (see (\ref{eq:coulomb_lbd_w})),
\begin{eqnarray}
 &  & \mathfrak{F}\!\left(\left\{ \frac{\gamma_{k}{}^{2}}{\lambda_{k}-\zeta^{-1}}\right\} _{k=1}^{\infty}\right)\label{eq:calF_eq_Coulomb_wave}\\
\noalign{\smallskip} &  & =\,\frac{\Gamma\!\left(\frac{1}{2}+\mu-\frac{1}{2}\sqrt{1+4\nu\zeta}\,\right)\Gamma\!\left(\frac{1}{2}+\mu+\frac{1}{2}\sqrt{1+4\nu\zeta}\,\right)}{\Gamma(\mu)\Gamma(\mu+1)}\, e^{-i\zeta}\,_{1}F_{1}(\mu+i\nu;2\mu;2i\zeta).\nonumber 
\end{eqnarray}
\end{proposition}

%\begin{proof}[Proof of (\ref{eq:calF_eq_Coulomb_wave})]

\begin{proof} Observe that the convergence condition (\ref{eq:assum_sum_w})
is satisfied in this example. For $n\in\mathbb{N}$ put
\[
f_{1,n}=\mathfrak{F}\!\left(\left\{ \frac{\gamma_{k}{}^{2}}{\lambda_{k}-\zeta^{-1}}\right\} _{\! k=n}^{\!\infty}\right)\!,
\]
and let $f_{2,n}$ be equal to the RHS of (\ref{eq:calF_eq_Coulomb_wave})
where we replace $\mu$ by $\mu+n-1$. According to (\ref{eq:F_T_recur}),
the sequence $\{f_{1,n}\}$ obeys the recurrence rule
\begin{equation}
f_{1,n}-f_{1,n+1}+X(\mu+n)f_{1,n+2}=0,\ n\in\mathbb{N},\label{eq:aux_Coulomb_recurr}
\end{equation}
where
\begin{eqnarray*}
X(x) & = & \frac{w(x,\nu)^{2}}{\left(\lambda(x,\nu)-\zeta^{-1}\right)\left(\lambda(x+1,\nu)-\zeta^{-1}\right)}\\
\noalign{\smallskip} & = & \frac{(x^{2}-1)(x^{2}+\nu^{2})\zeta^{2}}{(4x^{2}-1)\left((x-1)x-\nu\zeta\right)\left(x(x+1)-\nu\zeta\right)}\,\ \text{for}\ x>1.
\end{eqnarray*}
Next one can apply the identity
\begin{equation}
\,_{1}F_{1}(a-1;b-2;z)-\frac{b^{2}-2b+(2a-b)z}{(b-2)b}\,_{1}F_{1}(a;b;z)-\frac{a(b-a)z^{2}}{\left(b^{2}-1\right)b^{2}}\,_{1}F_{1}(a+1;b+2;z)=0,\label{eq:recurr_1F1_ab}
\end{equation}
as it follows from \cite[\S13.4]{AbramowitzStegun}, to verify that
the sequence $\{f_{2,n}\}$ obeys (\ref{eq:aux_Coulomb_recurr}) as
well. Notice that, if rewritten in terms of Coulomb wave functions,
(\ref{eq:recurr_1F1_ab}) amounts to the recurrence rule \cite[Eq.~14.2.3]{AbramowitzStegun}
\[
L\sqrt{(L+1)^{2}+\eta^{2}}\, u_{L+1}-(2L+1)\!\left(\eta+\frac{L(L+1)}{\rho}\right)\! u_{L}+(L+1)\sqrt{L^{2}+\eta^{2}}\, u_{L-1}=0,
\]
where $u_{L}=F_{L}(\eta,\rho)$.

To evaluate the limit of $f_{2,n}$, as $n\to\infty$, one may notice
that
\[
\lim_{n\to\infty}\,_{1}F_{1}(a+n;b+\kappa n;z)=e^{z/\kappa}
\]
for $\kappa\neq0$, and apply the Stirling formula. Alternatively,
avoiding the Stirling formula, the limit is also obvious from the
identity \cite[Eq.~8.325(1)]{GradshteynRyzhik}
\begin{equation}
\prod_{k=0}^{\infty}\left(1+\frac{z}{(y+k)(y+k+1)}\right)=\frac{\Gamma(y)\Gamma(y+1)}{\Gamma\left(\frac{1}{2}+y-\frac{1}{2}\sqrt{1-4z}\right)\Gamma\left(\frac{1}{2}+y+\frac{1}{2}\sqrt{1-4z}\right)}\,.\label{eq:prod_k-sq_eq_Gammas}
\end{equation}
In any case, $\lim_{n\to\infty}f_{2,n}=1$ and so, in virtue of Lemma~\ref{lem:Fn_eq_calF},
$f_{1,n}=f_{2,n}$, $\forall n$. In particular, for $n=1$ one gets
(\ref{eq:calF_eq_Coulomb_wave}). \end{proof}

\begin{proof}[Proof of formulas (\ref{eq:specJ_Cwave}) and (\ref{eq:eigenvJ_Cwave})]
As recalled in Section~\ref{sec:Preliminaries} (see (\ref{eq:specJ_xi0})),
$z=\zeta^{-1}\neq0$ is an eigenvalue of $J(\mu,\nu)$ if and only
if $\xi_{0}(z)=0$ which means nothing but (\ref{eq:specJ_Cwave}).
In that case the components $\xi_{n}(z)$, $n\in\mathbb{N}$, of a
corresponding eigenvector can be chosen as described in (\ref{eq:def_xi_k}).
Note that
\[
\prod_{k=1}^{n-1}w_{k}=2^{n-1}\sqrt{(2\mu+1)(2\mu+2n-1)}\left|\frac{\Gamma(\mu+n+i\nu)}{\Gamma(\mu+1+i\nu)}\right|\frac{\Gamma(2\mu+1)}{\Gamma(2\mu+2n)}\,
\]
and that (\ref{eq:prod_k-sq_eq_Gammas}) means in fact the equality
\[
\prod_{k=1}^{\infty}\frac{1}{1-\lambda(\mu+k,\nu)\, z}=\frac{\Gamma\!\left(\frac{1}{2}+\mu-\frac{1}{2}\sqrt{1+4\nu z}\right)\Gamma\!\left(\frac{1}{2}+\mu+\frac{1}{2}\sqrt{1+4\nu z}\right)}{\Gamma(\mu)\Gamma\,(\mu+1)}\,.
\]
Using these equations and omitting a constant factor one finally arrives
at formula (\ref{eq:eigenvJ_Cwave}). \end{proof}

\section{Confluent hypergeometric functions\label{sec:Confluent-hypergeometric-fun}}

First, let us show an identity.

\begin{proposition} The equation
\begin{eqnarray}
 &  & \frac{\Gamma(x+\gamma+n)}{\Gamma(x+\gamma)}\, e^{x}\,\mathfrak{F}\!\left(\left\{ \frac{\sqrt{2x}\,\Gamma\!\left(\frac{1}{2}(\gamma-\alpha+k+1)\right)}{(x+\gamma+k-1)\,\Gamma\!\left(\frac{1}{2}(\gamma-\alpha+k)\right)}\right\} _{\! k=1}^{\! n}\right)\nonumber \\
\noalign{\smallskip} &  & =\,\frac{\Gamma(\gamma+n)}{\Gamma(\gamma)}\,\,_{1}F_{1}(\alpha;\gamma;x)\,_{1}F_{1}(\alpha-\gamma-n;1-\gamma-n;x)\label{eq:FF_eq_calF}\\
 &  & \quad-\,\frac{\Gamma(\gamma-1)\Gamma(\gamma-\alpha+n+1)}{\Gamma(\gamma-\alpha)\Gamma(\gamma+n+1)}\, x^{n+1}\,\,_{1}F_{1}(\alpha-\gamma+1;2-\gamma;x)\,_{1}F_{1}(\alpha;\gamma+n+1;x),\nonumber 
\end{eqnarray}
is valid for $\alpha,\gamma,x\in\mathbb{C}$ and $n\in\mathbb{Z}_{+}$
(if considering the both sides as meromorphic functions). \end{proposition}

\begin{remark} For instance, as a particular case of (\ref{eq:FF_eq_calF})
one gets, for $n=0$,
\begin{equation}
\,_{1}F_{1}(\alpha;\gamma;x)\,_{1}F_{1}(\alpha-\gamma;1-\gamma;x)-\frac{(\gamma-\alpha)x}{\gamma(\gamma-1)}\,\,_{1}F_{1}(\alpha;\gamma+1;x)\,_{1}F_{1}(\alpha-\gamma+1;2-\gamma;x)=e^{x}.\label{eq:FF_eq_exp}
\end{equation}
\end{remark}

%\begin{proof}[Proof of (\ref{eq:FF_eq_calF})]

\begin{proof} For $\alpha$, $\gamma$ and $x$ fixed and $n\in\mathbb{Z}$,
put
\[
\varphi_{n}=\frac{1}{\Gamma(n+\gamma)}\,\,_{1}F_{1}(\alpha;n+\gamma;x),\text{ }\psi_{n}=\frac{1}{\Gamma(n+\gamma-\alpha)}\, U(\alpha,n+\gamma,x).
\]
Then $\left\{ \varphi_{n}\right\} $ and $\left\{ \psi_{n}\right\} $
obey the second-order difference equation \cite[Eqs.~13.4.2,13.4.16]{AbramowitzStegun}
\begin{equation}
(n+\gamma-\alpha)xu_{n+1}-(n+\gamma+x-1)u_{n}+u_{n-1}=0,\ \ n\in\mathbb{Z}.\label{eq:deff_eq_confluent}
\end{equation}
Note also that
\[
\text{ }(\alpha-\gamma)\,_{1}F_{1}(\alpha;\gamma+1;x)U(\alpha,\gamma,x)+\gamma\,_{1}F_{1}(\alpha;\gamma;x)U(\alpha,\gamma+1,x)=\frac{\Gamma(\gamma+1)}{\Gamma(\alpha)}\, x^{-\gamma}e^{x}
\]
(as it follows, for example, from equations 13.4.12 and 13.4.25 combined
with 13.1.22 in \cite{AbramowitzStegun}). Whence
\[
\varphi_{0}\psi_{1}-\varphi_{1}\psi_{0}=\frac{1}{\Gamma(\alpha)\Gamma(1-\alpha+\gamma)}\, x^{-\gamma}e^{x},
\]
and so the solutions $\varphi_{n}$, $\psi_{n}$ are linearly independent
except of the cases $-\alpha\in\mathbb{Z}_{+}$ and $\alpha-\gamma\in\mathbb{N}$.

The difference equation (\ref{eq:deff_eq_confluent}) can be symmetrized
using the substitution
\[
u_{n}=\frac{x^{-n}}{\Gamma(\gamma-\alpha+n)}\, v_{n}.
\]
Then $w_{n}v_{n+1}-\zeta_{n}v_{n}+w_{n-1}v_{n-1}=0$ where
\[
w_{n}=\frac{x^{-n}}{\Gamma(\gamma-\alpha+n)}\,,\ \zeta_{n}=\frac{(x+\gamma+n-1)x^{-n}}{\Gamma(\gamma-\alpha+n)}\,.
\]
For a solution of the equation $\gamma_{n}\gamma_{n+1}=w_{n}$, $\forall n$,
one can take
\[
\gamma_{n}=2^{\frac{1}{4}}x^{-\frac{n}{2}+\frac{1}{4}}\sqrt{\frac{\Gamma\!\left(\frac{1}{2}(\gamma-\alpha+n+1)\right)}{\Gamma(\gamma-\alpha+n)\Gamma\!\left(\frac{1}{2}(\gamma-\alpha+n)\right)}}\,.
\]
Referring to another solution, namely
\[
g_{n}=\frac{1}{w_{n-1}\mathcal{P}_{n-1}}\,\mathfrak{F}\!\left(\left\{ \frac{\gamma_{k}^{\,2}}{\zeta_{k}}\right\} _{k=-\infty}^{n-1}\right)\!,\ \text{with}\ n\in\mathbb{N},
\]
using otherwise the same notation as in (\ref{eq:sols_diff_eq_inf}),
one concludes that there exist constants $A$ and $B$ such that
\[
\frac{\Gamma(x+\gamma+n)}{\Gamma(\gamma-\alpha+n+1)}\, x^{-n}\,\mathfrak{F}\!\left(\!\left\{ \frac{\sqrt{2x}\,\Gamma\!\left(\frac{1}{2}(\gamma-\alpha+k+1)\right)}{(x+\gamma+k-1)\,\Gamma\!\left(\frac{1}{2}(\gamma-\alpha+k)\right)}\right\} _{\! k=1}^{\! n}\right)\!=A\varphi_{n+1}+B\psi_{n+1}
\]
for all $n\in\mathbb{Z}_{+}$. $A$ and $B$ can be determined from
the values for $n=-1,0$ (putting $\mathfrak{F}\!\left(\left\{ x_{k}\right\} _{k=1}^{-1}\right)=0$,
as dictated by the recurrence rule (\ref{eq:gen_recurrence}) provided
the admissible values are extended to $N_{1}=N_{2}=1$, $n=0$). After
some manipulations one gets
\begin{eqnarray*}
 &  & \hskip-1.5em\frac{\Gamma(\gamma-\alpha)}{\Gamma(\gamma-\alpha+n+1)}\,\,_{1}F_{1}(\alpha;\gamma;x)U(\alpha,\gamma+n+1,x)\\
 &  & \hskip-1.5em-\,\frac{\Gamma(\gamma)}{\Gamma(\gamma+n+1)}\, U(\alpha,\gamma,x)\,_{1}F_{1}(\alpha;\gamma+n+1;x)\\
 &  & \hskip-1.5em=\frac{\Gamma(\gamma)\Gamma(\gamma-\alpha)\Gamma(x+\gamma+n)}{\Gamma(\alpha)\Gamma(\gamma-\alpha+n+1)\Gamma(x+\gamma)}\, x^{-\gamma-n}e^{x}\,\mathfrak{F}\!\left(\!\left\{ \!\frac{\sqrt{2x}\,\Gamma\!\left(\frac{1}{2}(\gamma-\alpha+k+1)\right)}{(x+\gamma+k-1)\Gamma\!\left(\frac{1}{2}(\gamma-\alpha+k)\right)}\right\} _{\! k=1}^{\! n}\right)\!\!.
\end{eqnarray*}
Recall that
\begin{equation}
U(a,b,x)=\frac{\Gamma(1-b)}{\Gamma(a-b+1)}\,_{1}F_{1}(a;b;x)+\frac{\Gamma(b-1)}{\Gamma(a)}\, x^{1-b}\,_{1}F_{1}(a-b+1;2-b;x),\label{eq:rel_1F1_U}
\end{equation}
whence (\ref{eq:FF_eq_calF}). \end{proof}

\begin{remark} Let us point out two particular cases of (\ref{eq:FF_eq_calF}).
Putting $\alpha=0$ one gets the identity
\begin{equation}
\mathfrak{F}\!\left(\left\{ \frac{\sqrt{2x}\,\Gamma\!\left(\frac{1}{2}(\gamma+k+1)\right)}{(x+\gamma+k-1)\Gamma\!\left(\frac{1}{2}(\gamma+k)\right)}\right\} _{\! k=1}^{\! n}\right)=\frac{\Gamma(x+\gamma)}{\Gamma(x+\gamma+n)}\,\sum_{j=0}^{n}\frac{\Gamma(\gamma+n-j)}{\Gamma(\gamma)}\, x^{j},\label{eq:calF_alpha_0}
\end{equation}
and for $\alpha=-1$ one obtains
\begin{eqnarray}
 &  & \mathfrak{F}\!\left(\left\{ \frac{\sqrt{2x}\,\Gamma\!\left(\frac{1}{2}(\gamma+k+2)\right)}{(x+\gamma+k-1)\Gamma\!\left(\frac{1}{2}(\gamma+k+1)\right)}\right\} _{k=1}^{n}\right)\nonumber \\
\noalign{\smallskip} &  & =\,\frac{\Gamma(x+\gamma)}{\Gamma(x+\gamma+n)}\,\sum_{j=0}^{n}\frac{\Gamma(\gamma+n-j)}{\Gamma(\gamma+1)}\,(\gamma-j(n-j))\, x^{j}.\label{eq:calF_alpha_-1}
\end{eqnarray}

Let us sketch a derivation of (\ref{eq:calF_alpha_-1}), equation
(\ref{eq:calF_alpha_0}) is simpler. Substitute $-1$ for $\alpha$
and $\gamma+n$ for $\gamma$ in (\ref{eq:FF_eq_exp}), and put
\[
B_{n}=\frac{\Gamma(\gamma+n-1)}{\gamma-x+n}\, x^{-n}\,_{1}F_{1}(-\gamma-n;2-\gamma;x).
\]
Then
\[
B_{n+1}-B_{n}=\frac{\Gamma(\gamma+n+1)}{(\gamma-x+n)(\gamma-x+n+1)}\, x^{-n-1}e^{x}.
\]
Whence
\[
B_{n+1}=B_{0}+e^{x}\,\sum_{j=1}^{n+1}\frac{\Gamma(\gamma+j)}{(\gamma-x+j-1)(\gamma-x+j)}\, x^{-j}
\]
which means nothing but
\begin{eqnarray}
 &  & \Gamma(\gamma+n)(\gamma-x)\,_{1}F_{1}(-\gamma-n-1;1-\gamma-n;x)\nonumber \\
 &  & -\,\Gamma(\gamma-1)(\gamma-x+n+1)\, x^{n+1}\,_{1}F_{1}(-\gamma;2-\gamma;x)\label{eq:intermed_F_eq_sum}\\
 &  & =\,(\gamma-x+n+1)(\gamma-x)\, e^{x}\,\sum_{j=1}^{n+1}\frac{\Gamma(\gamma+j)}{(\gamma-x+j-1)(\gamma-x+j)}\, x^{n+1-j}.\nonumber 
\end{eqnarray}
Set $\alpha=-1$ in (\ref{eq:FF_eq_calF}) and notice that $\,_{1}F_{1}(-1;b;x)=1-x/b$.
After some simplifications, a combination of thus obtained identity
with (\ref{eq:intermed_F_eq_sum}) gives (\ref{eq:calF_alpha_-1}).
\end{remark}

As an application of (\ref{eq:FF_eq_calF}) consider the Jacobi matrix
operator $J(\alpha,\beta,\gamma)$ depending on parameters $\alpha$,
$\beta$, $\gamma$, with $\beta>0$, $\gamma>0$ and $\alpha+\beta>0$,
as introduced in (\ref{eq:matrixJ}) where we put
\begin{equation}
\lambda_{k}=\gamma k,\ w_{k}=\sqrt{\alpha+\beta k}\,,\ k=1,2,3,\ldots.\label{eq:confluent_J}
\end{equation}
For the sequence $\gamma_{k}$ (fulfilling $\gamma_{k}\gamma_{k+1}=w_{k}$)
one can take
\[
\gamma_{k}{}^{2}=\sqrt{2\beta}\,\Gamma\!\left(\frac{1}{2}\left(\frac{\alpha}{\beta}+k+1\right)\right)\!\bigg/\Gamma\!\left(\frac{1}{2}\left(\frac{\alpha}{\beta}+k\right)\right)\!.
\]

Regarding the diagonal of $J(\alpha,\beta,\gamma)$ as an unperturbed
part and the off-diagonal elements as a perturbation one immediately
realizes that the matrix $J(\alpha.\beta,\gamma)$ determines a unique
semibounded self-adjoint operator in $\ell^{2}(\mathbb{N})$. Moreover,
the Weyl theorem about invariance of the essential spectrum tells
us that its spectrum is discrete and simple. Our goal here is to show
that one can explicitly construct a ``characteristic'' function
of this operator in terms of confluent hypergeometric functions.

\begin{proposition}\label{prop:confluent_spec_J} The spectrum of
$J(\alpha,\beta,\gamma)$ defined in (\ref{eq:matrixJ}) and (\ref{eq:confluent_J})
coincides with the set of zeros of the function
\begin{equation}
F_{J}(\alpha,\beta,\gamma;z)=\,_{1}F_{1}\!\left(1-\frac{\alpha}{\beta}-\frac{\beta}{\gamma^{2}}-\frac{z}{\gamma};1-\frac{\beta}{\gamma^{2}}-\frac{z}{\gamma};\frac{\beta}{\gamma^{2}}\right)\!\bigg/\Gamma\!\left(1-\frac{\beta}{\gamma^{2}}-\frac{z}{\gamma}\right)\!.\label{eq:FJ_eq_1F1}
\end{equation}
Moreover, if $z$ is an eigenvalue then the components of a corresponding
eigenvector $v$ can be chosen as 
\begin{eqnarray}
 &  & \hskip-1emv_{k}\,=\,(-1)^{k}\beta^{k/2}\gamma^{-k}\,\frac{\Gamma\!\left(\frac{\alpha}{\beta}+k\right)^{1/2}}{\Gamma\!\left(1-\frac{\beta}{\gamma^{2}}-\frac{z}{\gamma}+k\right)}\,_{1}F_{1}\!\left(1-\frac{\alpha}{\beta}-\frac{\beta}{\gamma^{2}}-\frac{z}{\gamma};1-\frac{\beta}{\gamma^{2}}-\frac{z}{\gamma}+k;\frac{\beta}{\gamma^{2}}\right)\!,\nonumber \\
 &  & \hskip-1emk\in\mathbb{N}.\label{eq:f_k_eq_1F1}
\end{eqnarray}
\end{proposition}

\begin{remark} (i)~In principle it would be sufficient to consider
the case $\gamma=1$; observe that
\[
F_{J}(\alpha,\beta,\gamma;z)=F_{J}\!\left(\frac{\alpha}{\gamma^{2}},\frac{\beta}{\gamma^{2}},1;\frac{z}{\gamma}\right)\!.
\]
Thus for $\gamma=1$ we get a simpler expression,
\[
F_{J}(\alpha,\beta,1;z)=\,_{1}F_{1}\!\left(1-\frac{\alpha}{\beta}-\beta-z;1-\beta-z;\beta\right)\!\bigg/\Gamma(1-\beta-z).
\]

(ii)~Notice that the convergence condition (\ref{eq:assum_sum_w})
is violated in this example. \end{remark}

Before the proof we consider analogous results for finite matrices.
Let $J_{n}(\alpha,\beta,\gamma)$ be the principal $n\times n$ submatrix
of $J(\alpha,\beta,\gamma)$. The characteristic polynomial $F_{J_{n}}(z)$
of $J_{n}(\alpha,\beta,\gamma)$ can be expressed in terms of confluent
hypergeometric functions, too. According to (\ref{eq:charpoly}),
\[
F_{J_{n}}(\alpha,\beta,\gamma;z)=\gamma^{n}\,\frac{\Gamma\!\left(1-\frac{z}{\gamma}+n\right)}{\Gamma\!\left(1-\frac{z}{\gamma}\right)}\,\mathfrak{F}\!\left(\left\{ \frac{\sqrt{2\beta}}{\gamma}\frac{\Gamma\!\left(\frac{1}{2}\left(\frac{\alpha}{\beta}+k+1\right)\right)}{\left(k-\frac{z}{\gamma}\right)\Gamma\!\left(\frac{1}{2}\left(\frac{\alpha}{\beta}+k\right)\right)}\right\} _{\! k=1}^{\! n}\right)\!.
\]
Applying (\ref{eq:FF_eq_calF}) one arrives at the expression
\begin{eqnarray*}
 &  & F_{J_{n}}(\alpha,\beta,\gamma;z)\\
 &  & =\,\gamma^{n}\, e^{-\frac{\beta}{\gamma^{2}}}\Bigg(\,\frac{\Gamma\!\left(n+1-\frac{\beta}{\gamma^{2}}-\frac{z}{\gamma}\right)}{\Gamma\!\left(1-\frac{\beta}{\gamma^{2}}-\frac{z}{\gamma}\right)}\,\,_{1}F_{1}\!\left(1-\frac{\alpha}{\beta}-\frac{\beta}{\gamma^{2}}-\frac{z}{\gamma};1-\frac{\beta}{\gamma^{2}}-\frac{z}{\gamma};\frac{\beta}{\gamma^{2}}\right)\\
 &  & \qquad\qquad\qquad\times\,_{1}F_{1}\!\left(-n-\frac{\alpha}{\beta};-n+\frac{\beta}{\gamma^{2}}+\frac{z}{\gamma};\frac{\beta}{\gamma^{2}}\right)\\
 &  & \qquad-\left(\frac{\beta}{\gamma^{2}}\right)^{n+1}\frac{\Gamma\!\left(n+1+\frac{\alpha}{\beta}\right)\Gamma\!\left(-\frac{\beta}{\gamma^{2}}-\frac{z}{\gamma}\right)}{\Gamma\!\left(\frac{\alpha}{\beta}\right)\Gamma\!\left(n+2-\frac{\beta}{\gamma^{2}}-\frac{z}{\gamma}\right)}\,\,_{1}F_{1}\!\left(1-\frac{\alpha}{\beta};1+\frac{\beta}{\gamma^{2}}+\frac{z}{\gamma};\frac{\beta}{\gamma^{2}}\right)\\
 &  & \qquad\qquad\qquad\times\,_{1}F_{1}\!\left(1-\frac{\alpha}{\beta}-\frac{\beta}{\gamma^{2}}-\frac{z}{\gamma};n+2-\frac{\beta}{\gamma^{2}}-\frac{z}{\gamma};\frac{\beta}{\gamma^{2}}\right)\Bigg)\!.
\end{eqnarray*}

Eigenvectors can be explicitly expressed as well. If $z$ is an eigenvalue
of $J_{n}(\alpha,\beta,\gamma)$ then formula (\ref{eq:def_xi_k_meromorph})
admits adaptation to this situation giving the expression for the
components of a corresponding eigenvector,
\[
\xi_{k}^{(n)}=(-1)^{k-1}\left(\prod_{j=1}^{k-1}w_{j}\right)\left(\prod_{j=k+1}^{n}\left(\lambda_{j}-z\right)\right)\mathfrak{F}\!\left(\left\{ \frac{\gamma_{j}{}^{2}}{\lambda_{j}-z}\right\} _{\! j=k+1}^{\! n}\right)\!,\ k=1,2,\ldots,n.
\]
Notice that $\xi_{k}^{(n)}$ makes sense also for $k=n+1$, and in
that case its value is $0$. Using (\ref{eq:FF_eq_calF}) and omitting
a redundant constant factor one arrives after some straightforward
computation at the formula for an eigenvector $v^{(n)}$ of $J_{n}(\alpha,\beta,\gamma)$:
\begin{eqnarray*}
 &  & \hskip-1.5emv_{k}^{(n)}=\,(-1)^{k}\beta^{k/2}\gamma^{-k}\,\frac{1}{\sqrt{\Gamma\left(\frac{\alpha}{\beta}+k\right)}}\Bigg(\,\frac{\Gamma\!\left(\frac{\alpha}{\beta}+k\right)\Gamma\!\left(1-\frac{\beta}{\gamma^{2}}-\frac{z}{\gamma}+n\right)}{\Gamma\!\left(1-\frac{\beta}{\gamma^{2}}-\frac{z}{\gamma}+k\right)}\\
 &  & \qquad\times\,_{1}F_{1}\!\left(1-\frac{\alpha}{\beta}-\frac{\beta}{\gamma^{2}}-\frac{z}{\gamma};1-\frac{\beta}{\gamma^{2}}-\frac{z}{\gamma}+k;\frac{\beta}{\gamma^{2}}\right)\,_{1}F_{1}\!\left(-\frac{\alpha}{\beta}-n;\frac{\beta}{\gamma^{2}}+\frac{z}{\gamma}-n;\frac{\beta}{\gamma^{2}}\right)\\
 &  & \quad-\left(\frac{\beta}{\gamma^{2}}\right)^{\! n-k+1}\,\frac{\Gamma\!\left(1+\frac{\alpha}{\beta}+n\right)\Gamma\!\left(-\frac{\beta}{\gamma^{2}}-\frac{z}{\gamma}+k\right)}{\Gamma\!\left(2-\frac{\beta}{\gamma^{2}}-\frac{z}{\gamma}+n\right)}\\
 &  & \qquad\times\,_{1}F_{1}\!\left(1-\frac{\alpha}{\beta}-k;1+\frac{\beta}{\gamma^{2}}+\frac{z}{\gamma}-k;\frac{\beta}{\gamma^{2}}\right)\\
 &  & \qquad\times\,_{1}F_{1}\!\left(1-\frac{\alpha}{\beta}-\frac{\beta}{\gamma^{2}}-\frac{z}{\gamma};2-\frac{\beta}{\gamma^{2}}-\frac{z}{\gamma}+n;\frac{\beta}{\gamma^{2}}\right)\!\Bigg)\!,\ 1\leq k\leq n.
\end{eqnarray*}

\begin{remark} Formula (\ref{eq:f_k_eq_1F1}) can be derived informally
using a limit procedure. Suppose $z$ is an eigenvalue of the infinite
Jacobi matrix $J(\alpha,\beta,\gamma)$. For $k\in\mathbb{N}$ fixed,
considering the asymptotic behavior of $v_{k}^{(n)}$ as $n\to\infty$
one expects that the leading term may give the component $v_{k}$
of an eigenvector corresponding to the eigenvalue $z$. Omitting some
constant factors one actually arrives in this way at (\ref{eq:f_k_eq_1F1}).
But having in hand the explicit expressions (\ref{eq:FJ_eq_1F1})
and (\ref{eq:f_k_eq_1F1}) it is straightforward to verify directly
that the former one represents a characteristic function while the
latter one describes an eigenvector. \end{remark}

%\begin{proof}[Proof of formulas (\ref{eq:FJ_eq_1F1}) and (\ref{eq:f_k_eq_1F1})]

\begin{proof}[Proof of Proposition~\ref{prop:confluent_spec_J}] Observe
first that for $k=0$ the RHS of (\ref{eq:f_k_eq_1F1}) equals, up
to a constant factor, to the announced characteristic function (\ref{eq:FJ_eq_1F1}).
If $z$ solves the equation $v_{0}=0$ then one can make use of the
identity \cite[Eq.~13.4.2]{AbramowitzStegun}
\[
b(b-1)\,_{1}F_{1}(a;b-1;x)+b(1-b-x)\,_{1}F_{1}(a;b;x)+(b-a)x\,_{1}F_{1}(a;b+1;x)=0
\]
to verify that $v\in\ell^{2}(\mathbb{N})$ actually fulfills the eigenvalue
equation (\ref{eq:eigenvalue}). Note that the Stirling formula tells
us that
\[
v_{k}=\frac{(-1)^{k}}{(2\pi)^{1/4}}\, k^{-\frac{3}{4}+\frac{\alpha}{2\beta}+\frac{\beta}{\gamma^{2}}+\frac{z}{\gamma}}\left(\frac{\beta e}{\gamma^{2}k}\right)^{\! k/2}\left(1+O\!\left(\frac{1}{k}\right)\right)\text{ }\text{as}\text{ }k\to\infty.
\]
On the other hand, whatever the complex number $z$ is, the sequence
$v_{k}$, $k\in\mathbb{N}$, solves the second-order difference equation
(\ref{eq:2ndorder_diff_eq_param_z}), and in that case it is even
true that
\[
w_{0}v_{0}+\left(\lambda_{1}-z\right)v_{1}+w_{1}v_{2}=0.
\]
Let $g_{k}$, $k\in\mathbb{N}$, be any other independent solution
of (\ref{eq:2ndorder_diff_eq_param_z}). Since the Wronskian
\[
w_{k}\left(v_{k}g_{k+1}-v_{k+1}g_{k}\right)=\text{const}\neq0
\]
does not depend on $k$, and clearly $\lim_{k\to\infty}\, w_{k}v_{k}=\lim_{k\to\infty}\, w_{k}v_{k+1}=0$,
the sequence $g_{k}$ cannot be bounded in any neighborhood of infinity.
Hence, up to a multiplier, $\left\{ v_{k}\right\} $ is the only square
summable solution of (\ref{eq:2ndorder_diff_eq_param_z}). One concludes
that $z$ is an eigenvalue of $J(\alpha,\beta,\gamma)$ if and only
if $w_{0}v_{0}=0$ (which covers also the case $\alpha=0$). \end{proof}

\begin{remark} A second independent solution of (\ref{eq:2ndorder_diff_eq_param_z})
can be found explicitly. For example, this is the sequence
\[
g_{k}=(-1)^{k}\beta^{k/2}\gamma^{-k}\,\Gamma\!\left(\frac{\alpha}{\beta}+k\right)^{\!-1/2}U\!\left(1-\frac{\alpha}{\beta}-\frac{\beta}{\gamma^{2}}-\frac{z}{\gamma},1-\frac{\beta}{\gamma^{2}}-\frac{z}{\gamma}+k,\frac{\beta}{\gamma^{2}}\right)\!,\text{ }k\in\mathbb{N},
\]
as it follows from the identity \cite[Eq.~13.4.16]{AbramowitzStegun}
\[
(b-a-1)U(a,b-1,x)+(1-b-x)U(a,b,x)+xU(a,b+1,x)=0.
\]
But using once more relation (\ref{eq:rel_1F1_U}) one may find as
a more convenient the solution

\begin{eqnarray*}
g_{k} & = & \left(\frac{\beta}{\gamma^{2}}\right)^{\!\frac{\beta}{\gamma^{2}}+\frac{z}{\gamma}-\frac{k}{2}}\frac{1}{\sqrt{\Gamma\!\left(\frac{\alpha}{\beta}+k\right)}\,\Gamma\!\left(1+\frac{\beta}{\gamma^{2}}+\frac{z}{\gamma}-k\right)}\\
 &  & \times\,\,_{1}F_{1}\!\left(1-\frac{\alpha}{\beta}-k;1+\frac{\beta}{\gamma^{2}}+\frac{z}{\gamma}-k;\frac{\beta}{\gamma^{2}}\right)\!,\ k\in\mathbb{N}.
\end{eqnarray*}
\end{remark}

\begin{remark} Let us point out that for $\alpha=0$ one gets a nontrivial
example of an unbounded Jacobi matrix operator whose spectrum is known
fully explicitly. In that case
\[
\lambda_{k}=\gamma k,\ w_{k}=\sqrt{\beta k}\,,\ k=1,2,3,\ldots,
\]
and
\[
F_{J}(0,\beta,\gamma;z)=e^{\beta/\gamma^{2}}\bigg/\Gamma\!\left(1-\frac{\beta}{\gamma^{2}}-\frac{z}{\gamma}\right)\!.
\]
Hence
\[
\text{spec}\, J(0,\beta,\gamma)=\left\{ -\frac{\beta}{\gamma}+\gamma j;\ j=1,2,3,\ldots\right\} \!.
\]
\end{remark}

\begin{remark} Finally we remark that another particular case of
interest is achieved in the formal limit $\beta\to0$. Set $\alpha=w^{2}$
for some $w>0$. Since \cite[Eq.~13.3.2]{AbramowitzStegun}
\[
\lim_{a\to\infty}\,\,_{1}F_{1}\!\left(a;b;-\frac{z}{a}\right)=z^{(1-b)/2}\,\Gamma(b)J_{b-1}(2\sqrt{z})
\]
one finds that
\[
\lim_{\beta\to0}F_{J}(w^{2},\beta,\gamma;z)=\left(\frac{w}{\gamma}\right)^{\! z/\gamma}J_{-z/\gamma}\!\left(\frac{2w}{\gamma}\right)\!.
\]
It is known for quite a long time \cite{GardZakrajsek,IkebeKikuchiFujishiro}
that actually
\[
\spec J(w^{2},0,\gamma)=\left\{ z\in\mathbb{C};\ J_{-z/\gamma}\!\left(\frac{2w}{\gamma}\right)=0\right\} \!.
\]
\end{remark}

\section{Q-Bessel functions\label{sec:Q-Bessel-functions}}

\subsection{Some properties of q-Bessel functions}

Here we aim to explore a q-analogue to the following well known property
of Bessel functions. Consider the eigenvalue problem
\[
wx_{k-1}-kx_{k}+wx_{k+1}=\nu x_{k},\text{ }k\in\mathbb{Z},
\]
for a second order difference operator acting in $\ell^{2}(\mathbb{Z})$
and depending on a parameter $w>0$. If $\nu\notin\mathbb{Z}$ then
one can take $\{J_{\nu+k}(2w)\}$ and $\{(-1)^{k}J_{-\nu-k}(2w)\}$
for two independent solutions of the formal eigenvalue equation while
for $\nu\in\mathbb{Z}$ this may be the couple $\{J_{\nu+k}(2w)\}$
and $\{Y_{\nu+k}(2w)\}$. Taking into account the asymptotic behavior
of Bessel functions for large orders (see \cite[Eqs.~9.3.1, 9.3.2]{AbramowitzStegun})
one finds that a square summable solution exists if and only if $\nu\in\mathbb{Z}$.
Then $x_{k}=J_{\nu+k}(2w)$, $k\in\mathbb{Z}$, is such a solution
and is unique up to a constant multiplier. Since
\begin{equation}
\sum_{k=-\infty}^{\infty}J_{k}(z){}^{2}=1,\label{eq:Bessel_sum2}
\end{equation}
thus obtained eigenbasis $\pmb{v}_{\nu}=\{v_{\nu,k}\}_{k=-\infty}^{\infty}$,
$\nu\in\mathbb{Z}$, with $v_{\nu,k}=J_{\nu+k}(2w)$, is even orthonormal.
One observes that the spectrum of the difference operator is stable
and equals $\mathbb{Z}$ independently of the parameter $w$.

Hereafter we assume $0<q<1$. Recall the second definition of the
q-Bessel function introduced by Jackson \cite{Jackson} (for some
basic information and references one can also consult \cite{GasperRahman}),
\[
J_{\nu}^{(2)}(x;q)=\frac{(q^{\nu+1};q){}_{\infty}}{(q;q)_{\infty}}\left(\frac{x}{2}\right)^{\nu}\,_{0}\phi_{1}\!\!\left(;q^{\nu+1};q,-\frac{q^{\nu+1}x^{2}}{4}\right)\!.
\]
Here we prefer a slight modification of the second q-Bessel function,
obtained just by some rescaling, and define 
\begin{equation}
\hskip-0.3em\mathfrak{j}_{\nu}(x;q):=q^{\left.\nu^{2}\right/4}J_{\nu}^{(2)}(q^{1/4}x;q)=q^{\nu(\nu+1)/4}\,\frac{(q^{\nu+1};q){}_{\infty}}{(q;q)_{\infty}}\,\left(\frac{x}{2}\right)^{\!\nu}\!\,_{0}\phi_{1}\!\!\left(;q^{\nu+1};q,-q^{\nu+3/2}\,\frac{x^{2}}{4}\right)\!.\label{eq:def_j_nu}
\end{equation}

With our definition we have the following property.

\begin{lemma} For every $n\in\mathbb{N}$,
\begin{equation}
\mathfrak{j}_{-n}(x;q)=(-1)^{n}\,\mathfrak{j}_{n}(x;q).\label{eq:j_-n_eq_j_n}
\end{equation}
\end{lemma}

\begin{proof} One can readily verify that
\[
\lim_{\nu\to-n}\,\left(1-q^{\nu+n}\right)\,_{0}\phi_{1}(;q^{\nu+1};q,-q^{\nu+3/2}w^{2})=-\frac{q^{\left.n^{2}\right/2}w^{2n}}{(q;q)_{n-1}(q;q)_{n}}\,\,_{0}\phi_{1}(;q^{n+1};q,-q^{n+3/2}w^{2})
\]
and
\[
\lim_{\nu\to-n}\,\frac{(q^{\nu+1};q){}_{\infty}}{1-q^{\nu+n}}=(-1)^{n-1}q^{-n(n-1)/2}(q;q)_{n-1}(q;q){}_{\infty}.
\]
The lemma is an immediate consequence. \end{proof}

\begin{proposition}\label{prop:qBessel_calF} For $0<q<1$, $w,\nu\in\mathbb{C}$,
$q^{-\nu}\notin q^{\mathbb{Z}_{+}}$, one has
\begin{equation}
\mathfrak{F}\!\left(\left\{ \frac{w}{q^{-(\nu+k)/2}-q^{(\nu+k)/2}}\right\} _{\! k=0}^{\!\infty}\right)=\,_{0}\phi_{1}(;q^{\nu};q,-q^{\nu+1/2}w^{2}).\label{eq:calF_eq_q-confluent}
\end{equation}
\end{proposition}

\begin{remark} If rewritten in terms of q-Bessel functions, (\ref{eq:calF_eq_q-confluent})
becomes a q-analogue of (\ref{eq:BesselJ_rel_F}). Explicitly, 
\[
\mathfrak{F}\!\left(\left\{ \frac{w}{[\nu+k]_{q}}\right\} _{\! k=1}^{\!\infty}\right)=q^{\nu/4}\,\Gamma_{q}(\nu+1)w^{-\nu}J_{\nu}^{(2)}(2q^{-1/4}(1-q)w;q)
\]
where \cite{GasperRahman}
\[
[x]_{q}=\frac{q^{x/2}-q^{-x/2}}{q^{1/2}-q^{-1/2}},\ \Gamma_{q}(x)=\frac{(q;q)_{\infty}}{(q^{x};q)_{\infty}}\,(1-q)^{1-x}.
\]
\end{remark}

\begin{lemma} For $\nu\in\mathbb{C}$, $q^{-\nu}\notin q^{\mathbb{Z}_{+}}$,
and all $s\in\mathbb{N}$,
\begin{equation}
\sum_{k=0}^{\infty}\frac{q^{sk}}{(q^{\nu+k};q)_{s+1}}=\frac{1}{\left(1-q^{s}\right)(q^{\nu};q)_{s}}\,.\label{eq:aux_q-confluent}
\end{equation}
\end{lemma}

%\begin{proof}[Proof of (\ref{eq:aux_q-confluent})]

\begin{proof} One can proceed by mathematical induction in $s$.
The identity
\[
\frac{q^{sk}}{(q^{\nu+k};q)_{s+1}}=\frac{q^{(s-1)k}}{q^{\nu}(1-q^{s})}\!\left(\frac{1}{(q^{\nu+k};q)_{s}}-\frac{1}{(q^{\nu+k+1};q)_{s}}\right)
\]
can be used to verify both the case $s=1$ and the induction step
$s\to s+1$. \end{proof}

%\begin{proof}[Proof of (\ref{eq:calF_eq_q-confluent})]

\begin{proof}[Proof of Proposition~\ref{prop:qBessel_calF}] One possibility
how to prove (\ref{eq:calF_eq_q-confluent}) is based on Lemma~\ref{lem:Fn_eq_calF}.
The proof presented below relies, however, on explicit evaluation
of the involved sums. For $\nu\in\mathbb{C}$, $q^{\nu}\notin q^{\mathbb{Z}}$,
$k\in\mathbb{Z}$, put
\[
\rho_{k}=\frac{q^{(\nu+k)/2}}{1-q^{\nu+k}}\,.
\]
Then (\ref{eq:aux_q-confluent}) immediately implies that, for $n\in\mathbb{Z}$
and $s\in\mathbb{N}$,
\[
\sum_{k=n}^{\infty}q^{(s-1)(\nu+k)/2}\rho_{k}\rho_{k+1}\ldots\rho_{k+s}=\frac{q^{s(\nu+n+1)/2}}{1-q^{s}}\,\rho_{n}\rho_{n+1}\ldots\rho_{n+s-1}.
\]
This equation in turn can be used in the induction step on $m$ to
show that, for $m\in\mathbb{N}$, $n\in\mathbb{Z}$,
\begin{eqnarray*}
 &  & \sum_{k_{1}=n}^{\infty}\sum_{k_{2}=k_{1}+2}^{\infty}\ldots\sum_{k_{m}=k_{m-1}+2}^{\infty}\rho_{k_{1}}\rho_{k_{1}+1}\rho_{k_{2}}\rho_{k_{2}+1}\ldots\rho_{k_{m}}\rho_{k_{m}+1}\\
\noalign{\smallskip} &  & =\,\frac{q^{m(3m+1)/4}\, q^{m(\nu+n-1)/2}}{(q;q)_{m}}\,\rho_{n}\rho_{n+1}\ldots\rho_{n+m-1}.
\end{eqnarray*}
In particular, for $n=1$ one gets
\[
\sum_{k_{1}=1}^{\infty}\sum_{k_{2}=k_{1}+2}^{\infty}\ldots\sum_{k_{m}=k_{m-1}+2}^{\infty}\rho_{k_{1}}\rho_{k_{1}+1}\rho_{k_{2}}\rho_{k_{2}+1}\ldots\rho_{k_{m}}\rho_{k_{m}+1}=\frac{q^{m(2m+1)/2+\nu m}}{(q;q)_{m}(q^{\nu+1};q)_{m}}\,,\ m\in\mathbb{N}.
\]
Now, in order to evaluate $\mathfrak{F}\!\left(\left\{ w\rho_{k}\right\} _{k=1}^{\infty}\right)$,
it suffices to apply the very definition (\ref{eq:def_calF}). \end{proof}

The q-hypergeometric function is readily seen to satisfy the recurrence
rule
\[
\,_{0}\phi_{1}(;q^{\nu};q,z)-\,_{0}\phi_{1}(;q^{\nu+1};q,qz)-\frac{z}{\left(1-q^{\nu}\right)\left(1-q^{\nu+1}\right)}\,\,_{0}\phi_{1}(;q^{\nu+2};q,q^{2}z)=0.
\]
Consequently,
\[
w\,\mathfrak{j}_{\nu}(2w;q)-\left(q^{-(\nu+1)/2}-q^{(\nu+1)/2}\right)\mathfrak{j}_{\nu+1}(2w;q)+w\,\mathfrak{j}_{\nu+2}(2w;q)=0.
\]
This is in agreement with (\ref{eq:sols_diff_eq_inf}) if applied
to the bilateral second order difference equation 
\begin{equation}
wx_{n-1}-\left(q^{-(\nu+n)/2}-q^{(\nu+n)/2}\right)x_{n}+wx_{n+1}=0,\ n\in\mathbb{Z}.\label{eq:bilateral_j_k}
\end{equation}
Suppose $q^{\nu}\notin q^{\mathbb{Z}}$. Then the two solutions described
in (\ref{eq:sols_diff_eq_inf}) in this case give
\begin{eqnarray}
f_{n} & = & q^{-\nu(\nu+1)/4}\,\frac{(q;q)_{\infty}}{(q^{\nu+1};q){}_{\infty}}\, w^{-\nu}\,\mathfrak{j}_{\nu+n}(2w;q),\label{eq:qBessel_sol_f}\\
g_{n} & = & (-1)^{n+1}q^{-\nu(\nu+1)/4}\,\frac{(q;q)_{\infty}}{(q^{-\nu};q){}_{\infty}}\, w^{\nu}\,\mathfrak{j}_{-\nu-n}(2w;q),\ n\in\mathbb{Z}.\label{eq:qBessel_sol_g}
\end{eqnarray}
Let us show that they are generically independent. For the proof we
need the identity \cite[\S1,3]{GasperRahman}
\[
\sum_{k=0}^{\infty}\,\frac{q^{\left.k(k-1)\right/2}}{(q;q)_{k}}\, z^{k}=(-z;q)_{\infty}.
\]

\begin{lemma} For $\nu\in\mathbb{C}$, $q^{\nu}\notin q^{\mathbb{Z}}$,
the Wronskian of the solutions of (\ref{eq:bilateral_j_k}), $\{f_{n}\}$
and $\{g_{n}\}$ defined in (\ref{eq:qBessel_sol_f}) and (\ref{eq:qBessel_sol_g}),
respectively, fulfills 
\begin{equation}
\mathcal{W}(f,g)=\mathfrak{F}\!\left(\left\{ \frac{w\, q^{(\nu+k)/2}}{1-q^{\nu+k}}\right\} _{\! k=-\infty}^{\!\infty}\right)\!=(-q^{1/2}w^{2};q)_{\infty}.\label{eq:Wronskian_qBessel}
\end{equation}
\end{lemma}

\begin{proof} The first equality in (\ref{eq:Wronskian_qBessel})
is nothing but (\ref{eq:Wronskian_fg}). Further, in virtue of (\ref{eq:calF_eq_q-confluent}),
the second member in (\ref{eq:Wronskian_qBessel}) equals
\begin{eqnarray*}
\lim_{N\to\infty}\,_{0}\phi_{1}(;q^{\nu-N};q,-q^{\nu-N+1/2}w^{2}) & = & \!\lim_{M\to\infty}\,\sum_{k=0}^{\infty}\,\frac{q^{k(k-1)}}{(q;q)_{k}(q^{-M};q){}_{k}}\left(-q^{-M}q^{1/2}w^{2}\right)^{k}\\
 & = & \!\sum_{k=0}^{\infty}\,\frac{q^{k^{2}/2}}{(q;q)_{k}}\, w^{2k}\,=\,(-q^{1/2}w^{2};q)_{\infty}.
\end{eqnarray*}
The lemma follows. \end{proof}

At the same time, $\mathcal{W}(f,g)$ equals
\[
\frac{q^{-\nu(\nu+1)/2}\,(q;q)_{\infty}{}^{\!2}\, w}{(q^{\nu+1};q){}_{\infty}(q^{-\nu};q){}_{\infty}}\left(\mathfrak{j}_{\nu}(2w;q)\,\mathfrak{j}_{-\nu-1}(2w;q)+\mathfrak{j}_{\nu+1}(2w;q)\,\mathfrak{j}_{-\nu}(2w;q)\right).
\]
This implies the following result.

\begin{proposition} For $w\in\mathbb{C}$ one has
\begin{eqnarray}
 &  & \mathfrak{j}_{\nu}(2w;q)\,\mathfrak{j}_{-\nu-1}(2w;q)+\mathfrak{j}_{\nu+1}(2w;q)\,\mathfrak{j}_{-\nu}(2w;q)\nonumber \\
\noalign{\medskip} &  & =\,\frac{q^{\nu(\nu+1)/2}\,(q^{\nu+1};q){}_{\infty}(q^{-\nu};q){}_{\infty}\,(-q^{1/2}w^{2};q)_{\infty}}{(q;q)_{\infty}{}^{2}\, w}\label{eq:qBessel_qBessel}
\end{eqnarray}
and, rewriting (\ref{eq:qBessel_qBessel}) in terms of q-hypergeometric
functions,
\begin{eqnarray}
 &  & \,_{0}\phi_{1}(;q^{\nu+1};q,-q^{\nu+1}z)\,_{0}\phi_{1}(;q^{-\nu};q,-q^{-\nu}z)\nonumber \\
 &  & -\frac{q^{\nu}z}{\left(1-q^{\nu}\right)\left(1-q^{\nu+1}\right)}\,\,_{0}\phi_{1}(;q^{\nu+2};q,-q^{\nu+2}z)\,_{0}\phi_{1}(;q^{-\nu+1};q,-q^{-\nu+1}z)\label{eq:0phi1_0phi1}\\
\noalign{\smallskip} &  & =\,(-z;q)_{\infty}.\nonumber 
\end{eqnarray}
\end{proposition}

\begin{remark} Let us examine the limit $q\to1-$ applied to (\ref{eq:qBessel_qBessel})
while replacing $w$ by $(1-q)w$. One knows that \cite{GasperRahman}
\[
\lim_{q\to1-}\,\mathfrak{j}_{\nu}((1-q)z;q)=J_{\nu}(z),\ \ \lim_{q\to1-}\,(1-q)^{1-x}\,\frac{(q;q)_{\infty}}{(q^{x};q){}_{\infty}}=\Gamma(x).
\]
Thus one finds that the limiting equation coincides with the well
known identity
\[
J_{\nu}(2w)J_{-\nu-1}(2w)+J_{\nu+1}(2w)J_{-\nu}(2w)=\frac{1}{w\Gamma(\nu+1)\Gamma(-\nu)}=-\frac{\sin(\pi\nu)}{\pi w}\,.
\]
\end{remark}

It is desirable to have some basic information about the asymptotic
behavior of q-Bessel functions for large orders. It is straightforward
to see that
\begin{equation}
\mathfrak{j}_{\nu}(x;q)=q^{\nu(\nu+1)/4}\,\frac{1}{(q;q)_{\infty}}\left(\frac{x}{2}\right)^{\nu}\left(1+O\left(q^{\nu}\right)\right)\text{ }\text{as}\ \Re\nu\to+\infty.\label{eq:asympt_j_nu_infty}
\end{equation}
The asymptotic behavior at $-\infty$ is described as follows.

\begin{lemma} For $\sigma,w\in\mathbb{C}$, $q^{\sigma}\notin q^{\mathbb{Z}}$,
one has
\begin{eqnarray}
 &  & \lim_{\substack{|\nu|\to\infty\cr\noalign{\smallskip}\nu\in-\sigma-\mathbb{N}}
}\,\sin(\pi\nu)\, q^{\nu(\nu+1)/4}\, w^{-\nu}\,\mathfrak{j}_{\nu}(2w;q)\nonumber \\
 &  & \quad=\,-\sin(\pi\sigma)\, q^{-\sigma(1-\sigma)/2}\,\frac{(q^{\sigma};q){}_{\infty}(q^{1-\sigma};q){}_{\infty}\,(-q^{1/2}w^{2};q)_{\infty}}{(q;q)_{\infty}}.\label{eq:lim_j_nu_-infty}
\end{eqnarray}
\end{lemma}

%\begin{proof}[Proof of (\ref{eq:lim_j_nu_-infty})]

\begin{proof} Put $\nu=-\sigma-n$ where $n\in\mathbb{N}$. Using
(\ref{eq:calF_eq_q-confluent}) and (\ref{eq:gen_recurrence}) one
can write
\begin{eqnarray*}
 &  & \,_{0}\phi_{1}(;q^{-\sigma-n};q,-q^{-\sigma-n+1/2}w^{2})\\
\noalign{\smallskip} &  & =\,\mathfrak{F}\!\left(\left\{ \frac{w}{q^{(\sigma+k)/2}-q^{-(\sigma+k)/2}}\right\} _{\! k=0}^{\! n}\right)\mathfrak{F}\!\left(\left\{ \frac{w}{q^{(\sigma-k)/2}-q^{-(\sigma-k)/2}}\right\} _{\! k=1}^{\!\infty}\right)\\
 &  & \quad+\,\frac{w^{2}}{(q^{\sigma/2}-q^{-\sigma/2})(q^{(1-\sigma)/2}-q^{-(1-\sigma)/2})}\\
 &  & \quad\ \ \times\,\mathfrak{F}\!\left(\left\{ \frac{w}{q^{(\sigma+k)/2}-q^{-(\sigma+k)/2}}\right\} _{\! k=1}^{\! n}\right)\mathfrak{F}\!\left(\left\{ \frac{w}{q^{(\sigma-k)/2}-q^{-(\sigma-k)/2}}\right\} _{\! k=2}^{\!\infty}\right)\!.
\end{eqnarray*}
Applying the limit $n\to\infty$ one obtains
\begin{eqnarray*}
 &  & \lim_{n\to\infty}\,_{0}\phi_{1}(;q^{-\sigma-n};q,-q^{-\sigma-n+1/2}w^{2})\\
\noalign{\smallskip} &  & =\,\,_{0}\phi_{1}(;q^{\sigma};q,-q^{\sigma+1/2}w^{2})\,_{0}\phi_{1}(;q^{1-\sigma};q,-q^{-\sigma+3/2}w^{2})\\
 &  & \quad+\,\frac{w^{2}}{(q^{\sigma/2}-q^{-\sigma/2})(q^{(1-\sigma)/2}-q^{-(1-\sigma)/2})}\\
 &  & \quad\ \ \times\,\,_{0}\phi_{1}(;q^{1+\sigma};q,-q^{\sigma+3/2}w^{2})\,_{0}\phi_{1}(;q^{2-\sigma};q,-q^{-\sigma+5/2}w^{2})\\
\noalign{\smallskip} &  & =\,(-q^{1/2}w^{2};q)_{\infty}.
\end{eqnarray*}
To get the last equality we have used (\ref{eq:0phi1_0phi1}). Notice
also that
\[
\lim_{n\to\infty}(-1)^{n}q^{(-\sigma-n)(-\sigma-n+1)/2}\,(q^{-\sigma-n+1};q)_{\infty}=q^{\sigma(\sigma-1)/2}\,(q^{\sigma};q)_{\infty}(q^{1-\sigma};q)_{\infty}.
\]
The limit (\ref{eq:lim_j_nu_-infty}) then readily follows. \end{proof}

Finally we establish an identity which can be viewed as a q-analogue
to (\ref{eq:Bessel_sum2}).

\begin{proposition} For $0<q<1$ and $w\in\mathbb{C}$ one has
\begin{equation}
\sum_{k=-\infty}^{\infty}q^{-k/2}\,\mathfrak{j}_{k}(2w;q){}^{2}=\mathfrak{j}_{0}(2w;q){}^{2}+\sum_{k=1}^{\infty}\left(q^{k/2}+q^{-k/2}\right)\mathfrak{j}_{k}(2w;q){}^{2}=(-q^{1/2}w^{2};q){}_{\infty}.\label{eq:sum_qBessel_sqrd}
\end{equation}
Equivalently, if rewritten in terms of q-Bessel functions,
\[
J_{0}^{(2)}(2w;q){}^{2}+\sum_{k=1}^{\infty}\left(q^{k/2}+q^{-k/2}\right)q^{\left.k^{2}\right/2}J_{k}^{(2)}(2w;q){}^{2}=(-w^{2};q){}_{\infty}.
\]
\end{proposition}

%\begin{proof}[Proof of (\ref{eq:sum_qBessel_sqrd})]

\begin{proof} In \cite[(1.20)]{Rahman} it is shown that
\[
\frac{J_{\nu}^{(2)}(2w;q){}^{2}}{(-w^{2};q){}_{\infty}}=\left(\frac{(q^{\nu+1};q){}_{\infty}}{(q;q)_{\infty}}\right)^{\!2}w^{2\nu}\,\,_{3}\phi_{2}(q^{\nu+\frac{1}{2}},-q^{\nu+\frac{1}{2}},-q^{\nu+1};q^{\nu+1},q^{2\nu+1};q,-w^{2}),
\]
and this can be rewritten as
\[
\,_{0}\phi_{1}(\text{};q^{\nu+1};q,-q^{\nu+1}x)^{2}=(-x;q)_{\infty}\,\,_{3}\phi_{2}(q^{\nu+\frac{1}{2}},-q^{\nu+\frac{1}{2}},-q^{\nu+1};q^{\nu+1},q^{2\nu+1};q,-x).
\]
Hence (\ref{eq:sum_qBessel_sqrd}) is equivalent to
\begin{eqnarray*}
 &  & \,_{3}\phi_{2}(q^{1/2},-q^{1/2},-q;q,q;q,-x)\\
\noalign{\smallskip} &  & +\,\sum_{k=1}^{\infty}\frac{\left(q^{-k/2}+q^{k/2}\right)q^{\left.k^{2}\right/2}}{(q;q)_{k}{}^{2}}\,\,_{3}\phi_{2}(q^{k+\frac{1}{2}},-q^{k+\frac{1}{2}},-q^{k+1};q^{k+1},q^{2k+1};q,-x)\, x^{k}\,=\,1.
\end{eqnarray*}
Looking at the power expansion in $x$ one gets, equivalently, a countable
system of equations, for $n=1,2,3,\ldots$,
\begin{eqnarray*}
 &  & \frac{(q^{1/2};q){}_{n}(-q^{1/2};q){}_{n}(-q;q)_{n}}{(q;q)_{n}{}^{\!3}}\\
 &  & +\,\sum_{k=1}^{n}(-1)^{k}\,\frac{\left(q^{-k/2}+q^{k/2}\right)q^{\left.k^{2}\right/2}}{(q;q)_{k}{}^{\!2}}\frac{(q^{k+1/2};q){}_{n-k}(-q^{k+1/2};q){}_{n-k}(-q^{k+1};q){}_{n-k}}{(q;q)_{n-k}(q^{k+1};q){}_{n-k}(q^{2k+1};q){}_{n-k}}\,=\,0.
\end{eqnarray*}
The equations can be brought to the form
\[
\frac{1}{(q;q)_{n}{}^{\!2}}+\sum_{k=1}^{n}(-1)^{k}\,\frac{q^{\left.k(k-1)\right/2}\left(1+q^{k}\right)}{(q;q)_{n+k}(q;q)_{n-k}}=0
\]
or, more conveniently,
\[
\sum_{j=0}^{2n}(-1)^{j}\frac{q^{-j(2n-j+1)/2}}{(q;q)_{2n-j}(q;q)_{j}}=0.
\]
This is true indeed since, for any $m\in\mathbb{Z}_{+}$,
\[
\sum_{j=0}^{m}(-1)^{j}\frac{(q;q)_{m}}{(q;q)_{m-j}(q;q)_{j}}\, q^{-j(m-j)/2}x^{j}=(q^{-(m-1)/2}x;q)_{m}=\prod_{k=0}^{m-1}\left(1-q^{-(m-1)/2+k}x\right).
\]
This concludes the proof. \end{proof}

\subsection{A bilateral second order difference equation}

We know that the sequence $u_{n}=\mathfrak{j}_{\nu+n}(2w;q)$ obeys
(\ref{eq:bilateral_j_k}). Applying the substitution $q^{-\nu-1}=z$,
$w=q^{\frac{\nu}{2}+\frac{1}{4}}\beta$, one finds that the sequence
\begin{eqnarray}
v_{n} & = & q^{-n/4}u_{n}\,=\, q^{-n/4}\,\mathfrak{j}_{\nu+n}\left(2q^{(2\nu+1)/4}\beta;q\right)\label{eq:qBessel_sol_v_n}\\
 & = & q^{\left.-\left(\nu^{2}+2\nu+2\right)\right/4}q^{(n-1)(n-2)/4}\,\frac{(q^{n}z^{-1};q){}_{\infty}}{(q;q)_{\infty}}\left(\frac{\beta}{z}\right)^{\nu+n}\,_{0}\phi_{1}(;q^{n}z^{-1};q,-q^{n}z^{-2}\beta^{2}),\nonumber 
\end{eqnarray}
fulfills
\begin{equation}
q^{(n-1)/2}\beta v_{n}+\left(q^{n}-z\right)v_{n+1}+q^{n/2}\beta v_{n+2}=0,\ n\in\mathbb{Z}.\label{eq:qBessel_bilat_eq0}
\end{equation}

\begin{remark} One can as well consider the unilateral second order
difference equation
\[
(1-z)v_{1}+\beta v_{2}=0,\ q^{(n-1)/2}\beta v_{n}+\left(q^{n}-z\right)v_{n+1}+q^{n/2}\beta v_{n+2}=0,\ n=1,2,3,\ldots.
\]
From (\ref{eq:asympt_j_nu_infty}) it can be seen that the sequence
$\{v_{n}\}$ given in (\ref{eq:qBessel_sol_v_n}) is square summable
over $\mathbb{N}$. Considering the Wronskian one also concludes that
any other linearly independent solution of (\ref{eq:qBessel_bilat_eq0})
cannot be bounded on any neighborhood of $+\infty$. Hence the sequence
$v_{n}$, $n\in\mathbb{N}$, solves the eigenvalue problem in $\ell^{2}(\mathbb{N})$
iff $v_{0}=0$, i.e. iff $\mathfrak{j}_{\nu}(2w;q)=0$. In terms of
the new parameters $\beta$, $z$ this condition becomes the characteristic
equation for an eigenvalue $z$,
\[
(z^{-1};q){}_{\infty}\,\,_{0}\phi_{1}(;z^{-1};q,-z^{-2}\beta^{2})=0.
\]
This example has already been treated in \cite[Sec.~4.1]{StampachStovicek2}.
\end{remark}

For the bilateral equation it may be more convenient shifting the
index by 1 in (\ref{eq:qBessel_bilat_eq0}). This is to say that we
are going to solve the equation
\begin{equation}
q^{(n-1)/2}\beta v_{n-1}+\left(q^{n}-z\right)v_{n}+q^{n/2}\beta v_{n+1}=0,\ n\in\mathbb{Z},\label{eq:qBessel_bilat_eq1}
\end{equation}
rather than (\ref{eq:qBessel_bilat_eq0}). Denote again by $J=J(\beta,q)$,
with $\beta\in\mathbb{R}$ and $0<q<1$, the corresponding matrix
operator in $\ell^{2}(\mathbb{Z})$. One knows, however, that $J(-\beta,q)$
and $J(\beta,q)$ are unitarily equivalent and so, if convenient,
one can consider just the values $\beta\geq0$. In equation (\ref{eq:qBessel_bilat_eq1}),
$z$ is playing the role of a spectral parameter. Using a notation
analogous to (\ref{eq:matrixJ}) (now for the bilateral case), this
means that
\begin{equation}
w_{n}=q^{n/2}\beta,\text{ }\lambda_{n}=q^{n},\text{ }\text{and}\text{ }\zeta_{n}:=z-\lambda_{n}=z-q^{n},\ n\in\mathbb{Z}.\label{eq:qBessel_w_lbd_dzt}
\end{equation}
Notice that for a sequence $\{\gamma_{n}\}$ obeying $\gamma_{n}\gamma_{n+1}=w_{n}$,
$\forall n\in\mathbb{Z}$, one can take
\[
\gamma_{2k-1}{}^{\!2}=q^{k-1},\text{ }\gamma_{2k}{}^{\!2}=q^{k}\beta^{2}.
\]

Since the sequence $\{w_{n}/(\lambda_{n}+1)\}$ is summable over $\mathbb{Z}$,
the Weyl theorem tells us that the essential spectrum of the self-adjoint
operator $J(\beta,q)$ contains just one point, namely $0$. Hence
all nonzero spectral points are eigenvalues.

\begin{proposition} For $0<q<1$ and $\beta>0$, the spectrum of
the Jacobi matrix operator $J(\beta,q)$ in $\mathbb{\ell}^{2}(\mathbb{Z})$,
as introduced above (see (\ref{eq:qBessel_w_lbd_dzt})), is pure point,
all eigenvalues are simple and 
\[
\spec_{p}J(\beta,q)=\left(-\beta^{2}q^{\mathbb{Z}_{+}}\right)\cup q^{\mathbb{Z}}.
\]
Eigenvectors $\pmb{v}_{m}^{(+)}$ corresponding to the eigenvalues
$q^{m}$, $m\in\mathbb{Z}$, can be chosen as $\pmb{v}_{m}^{(+)}=\left\{ v_{m,k}^{(+)}\right\} _{k=-\infty}^{\infty}$,
with
\[
v_{m,k}^{(+)}=q^{(m-k)/4}\,\mathfrak{j}_{-m+k}(2q^{-(2m+1)/4}\beta;q).
\]
They are normalized as follows:
\[
\left\Vert \pmb{v}_{m}^{(+)}\right\Vert {}^{2}=\sum_{k=-\infty}^{\infty}q^{-k/2}\,\mathfrak{j}_{k}(2q^{-(2m+1)/4}\beta;q)^{2}=(-q^{-m}\beta^{2};q){}_{\infty},\ \forall m\in\mathbb{Z}.
\]
Eigenvector \ensuremath{\pmb{v}_{m}^{(-)}}
 corresponding to the eigenvalues $-\beta^{2}q^{m}$, $m\in\mathbb{Z}_{+}$,
can be chosen as $\pmb{v}_{m}^{(-)}=\left\{ v_{m,k}^{(-)}\right\} _{k=-\infty}^{\infty}$,
with
\begin{eqnarray}
v_{m,k}^{(-)} & = & \frac{(-1)^{k}q^{k(k-4m-1)/4}}{(q;q)_{\infty}}\,\beta^{-k}\left(-q^{-m+k+1}\beta^{-2};q\right){}_{\infty}\nonumber \\
\noalign{\smallskip} &  & \times\,\,_{0}\phi_{1}\!\left(;-q^{-m+k+1}\beta^{-2};q,-q^{-2m+k+1}\beta^{-2}\right)\!.\label{eq:qBessel_eigenvec_minus}
\end{eqnarray}
\end{proposition}

\begin{remark} An expression for the norms of vectors $\pmb{v}_{m}^{(-)}$
can be found, too, 
\begin{eqnarray*}
\left\Vert \pmb{v}_{m}^{(-)}\right\Vert {}^{2} & = & (-1)^{m}q^{-m(3m+1)/2}\,\frac{\left(-q\beta^{-2};q\right){}_{\infty}\left(-q^{-m}\beta^{-2};q\right){}_{\infty}\left(-q^{m+1}\beta^{2};q\right){}_{\infty}}{\left(-q\beta^{2};q\right){}_{\infty}\left(q^{m+1};q\right){}_{\infty}}\\
\noalign{\smallskip} &  & \times\,\frac{\,_{0}\phi_{1}\left(\text{};-q\beta^{-2};q,-q^{-m+1}\beta^{-2}\right)}{\,_{0}\phi_{1}\left(\text{};-q\beta^{2};q,-q^{-m+1}\beta^{2}\right)}\,,\ \forall m\in\mathbb{Z}_{+}.
\end{eqnarray*}
But the formula is rather cumbersome and its derivation somewhat lengthy
and this is why we did not include it in the proposition and omit
its proof. \end{remark}

\begin{proof} We use the substitution $z=q^{-\nu}$ where $\nu$
is in general complex. The right hand sides in (\ref{eq:sols_diff_eq_inf})
can be evaluated using (\ref{eq:calF_eq_q-confluent}) and (\ref{eq:def_j_nu}).
Applying some easy simplifications one gets two solutions of (\ref{eq:qBessel_bilat_eq1}):
\[
v_{n}=q^{-(\nu+n)/4}\,\mathfrak{j}_{\nu+n}(2q^{(2\nu-1)/4}\beta;q),\ \tilde{v}_{n}=(-1)^{n}q^{-(\nu+n)/4}\,\mathfrak{j}_{-\nu-n}(2q^{(2\nu-1)/4}\beta;q),\text{ }n\in\mathbb{Z}.
\]

One can argue that in the bilateral case, too, all eigenvalues of
$J(\beta,q)$ are simple. In fact, the solution $\{v_{n}\}$ asymptotically
behaves as
\[
v_{n}=\frac{1}{(q;q)_{\infty}}\, q^{\frac{1}{4}\left(n^{2}+(4\nu-1)n+(3\nu-1)\nu\right)}\beta^{\nu+n}\left(1+O\left(q^{n}\right)\right)\text{ }\text{as}\ n\to+\infty.
\]
For any other independent solution $\{y_{n}\}$ of (\ref{eq:qBessel_bilat_eq1}),
$q^{n/2}\left(y_{n}v_{n+1}-y_{n+1}v_{n}\right)$ is a nonzero constant.
Obviously, such a sequence $\left\{ y_{n}\right\} $ cannot be bounded
on any neighborhood of $+\infty$. A similar argument applies to the
solution $\{\tilde{v}_{n}\}$ for $n$ large but negative. In particular,
one concludes that $z=q^{-\nu}$ is an eigenvalue of $J(\beta,q)$
if and only if $\{v_{n}\}$ and $\{\tilde{v}_{n}\}$ are linearly
dependent.

Using (\ref{eq:qBessel_qBessel}) one can derive a formula for the
Wronskian,
\begin{eqnarray*}
\mathcal{W}\left(v,\tilde{v}\right) & = & q^{k/2}\beta\left(v_{k}\tilde{v}_{k+1}-v_{k+1}\tilde{v}_{k}\right)\\
 & = & (-1)^{k+1}\beta q^{-(2\nu+1)/4}\big(\mathfrak{j}_{\nu+k}(2q^{(2\nu-1)/4}\beta;q)\mathfrak{j}_{-\nu-k-1}(2q^{(2\nu-1)/4}\beta;q)\\
 &  & \qquad\qquad\qquad\qquad\quad+\,\mathfrak{j}_{\nu+k+1}(2q^{(2\nu-1)/4}\beta;q)\mathfrak{j}_{-\nu-k}(2q^{(2\nu-1)/4}\beta;q)\big)\\
\noalign{\smallskip} & = & \frac{q^{\nu(\nu-3)/2}\left(q^{\nu};q\right){}_{\infty}\left(q^{1-\nu};q\right){}_{\infty}\left(-q^{\nu}\beta^{2};q\right){}_{\infty}}{(q;q)_{\infty}{}^{\!2}}\,.
\end{eqnarray*}
Thus $z$ is an eigenvalue if and only if either $\left(z^{-1};q\right){}_{\infty}\left(qz;q\right){}_{\infty}=0$
or $\left(-z^{-1}\beta^{2};q\right){}_{\infty}=0$. In the former
case $z\in q^{\mathbb{Z}}$, in the latter case $-z\in\beta^{2}q^{\mathbb{Z}_{+}}$.

Thus in the case of positive eigenvalues one can put $\nu=-m$, with
$m\in\mathbb{Z}$. With this choice, $\{v_{k}\}$ coincides with $\{v_{m,k}^{(+)}\}$.
Notice that then the linear dependence of the sequences $\{v_{k}\}$
and $\{\tilde{v}_{k}\}$ is also obvious from (\ref{eq:j_-n_eq_j_n}).
Normalization of the eigenvectors $\pmb{v}_{m}^{(+)}$ is a consequence
of (\ref{eq:sum_qBessel_sqrd}).

As far as the negative spectrum is concerned, one can put, for example,
$\tau=-(i\pi+\log\beta^{2})/\log q$ and $\nu=\tau-m$, $m\in\mathbb{Z}_{+}$.
Then the sequence
\[
v_{k}=q^{-(\tau-m+k)/4}\,\mathfrak{j}_{\tau-m+k}(-2iq^{-(2m+1)/4};q),\ k\in\mathbb{Z},
\]
represents an eigenvector corresponding to the eigenvalue $-\beta^{2}q^{m}$.
But it is readily seen to be proportional to the RHS of (\ref{eq:qBessel_eigenvec_minus})
whose advantage is to be manifestly real.

Finally let is show that $0$ can never be an eigenvalue of $J(\beta,q)$.
We still assume $\beta>0$. For $z=0$, one can find two mutually
complex conjugate solutions of (\ref{eq:qBessel_bilat_eq1}) explicitly.
Let us call them $v_{\pm,n}$, $n\in\mathbb{Z}$, where
\[
v_{\pm,n}=i^{\pm n}q^{-n/4}\,_{1}\phi_{1}\!\!\left(0;-q^{1/2};q^{1/2},\pm\frac{iq^{(2n+3)/4}}{\beta}\right)=i^{\pm n}q^{-n/4}\sum_{k=0}^{\infty}\frac{q^{k(k+2)/4}}{(q;q)_{k}}\left(\mp\frac{iq^{n/2}}{\beta}\right)^{\! k}\!.
\]
Clearly,
\[
v_{\pm,n}=i^{\pm n}q^{-n/4}\left(1+O\left(q^{n/2}\right)\right)\text{ }\text{as}\ n\to+\infty.
\]
Using the asymptotic expansion one can evaluate the Wronskian getting
\[
\mathcal{W}\left(v_{+},v_{-}\right)=q^{n/2}\beta\left(v_{+,n}v_{-,n+1}-v_{+,n+1}v_{-,n}\right)=-2iq^{-1/4}\beta.
\]
Hence the two solutions are linearly independent. It is also obvious
from the asymptotic expansion that no nontrivial linear combination
of these solutions can be square summable. Hence \ensuremath{0}
 cannot be an eigenvalue of \ensuremath{J(\beta,q)}
 whatever \ensuremath{\beta}
 is, and this concludes the proof. \end{proof}

So one observes that the positive part of the spectrum of $J(\beta,q)$
is stable and does not depend on the parameter $\beta$. This behavior
is very similar to what one knows from the non-deformed case. On the
other hand, there is an essentially new feature in the q-case when
a negative part of the spectrum emerges for $\beta\neq0$, and it
is even infinite-dimensional though it shrinks to zero with the rate
$\beta^{2}$ as $\beta$ tends to $0$.

\section{Q-confluent hypergeometric functions\label{sec:Q-confluent-hypergeometric-fun}}

In this section we deal with the q-confluent hypergeometric function
\[
\,_{1}\phi_{1}\left(a;b;q,q^{2}z\right)=\sum_{k=0}^{\infty}(-1)^{k}q^{k(k-1)/2}\,\frac{(a;q)_{k}}{(b;q)_{k}(q;q)_{k}}\, z^{k}\,.
\]
It can readily be checked to obey the recurrence rules
\begin{eqnarray}
 &  & \hskip-1.5em-\frac{q^{\alpha+\gamma}\left(1-q^{\gamma-\alpha+1}\right)}{\left(1-q^{\gamma}\right)\left(1-q^{\gamma+1}\right)}\, z\,_{1}\phi_{1}(q^{\alpha};q^{\gamma+2};q,q^{\gamma+2}z)-\left(1-\frac{q^{\gamma}\, z}{1-q^{\gamma}}\right)\,_{1}\phi_{1}(q^{\alpha};q^{\gamma+1};q,q^{\gamma+1}z)\nonumber \\
\noalign{\smallskip} &  & \hskip-1.5em+\,\,_{1}\phi_{1}(q^{\alpha};q^{\gamma};q,q^{\gamma}z)=0\label{eq:q-confluent_recurr1}
\end{eqnarray}
and
\begin{eqnarray*}
 &  & \,_{1}\phi_{1}(q^{\alpha-\gamma+1};q^{2-\gamma};q,z)+\frac{q\left(q-q^{\gamma}-q^{1-\gamma}+1\right)}{q^{\gamma}-q^{\alpha}}\,_{1}\phi_{1}(q^{\alpha-\gamma-1};q^{-\gamma};q,z)\\
\noalign{\smallskip} &  & -\left(\frac{q\left(q-q^{\gamma}-q^{1-\gamma}+1\right)}{q^{\gamma}-q^{\alpha}}+\frac{q-q^{\gamma}}{q^{\gamma}-q^{\alpha}}\, z\right)\,_{1}\phi_{1}(q^{\alpha-\gamma};q^{1-\gamma};q,z)=0.
\end{eqnarray*}
Put, for $n\in\mathbb{Z}$,
\begin{eqnarray}
 &  & \hskip-3.5em\varphi_{n}=(q^{n+\gamma};q){}_{\infty}\,_{1}\phi_{1}(q^{\alpha};q^{n+\gamma};q,-q^{n+\gamma}z),\label{eq:q-confluent_sol_phi}\\
\noalign{\smallskip} &  & \hskip-3.5em\psi_{n}=q^{-\alpha(n+\gamma)-(n+\gamma-1)(n+\gamma-2)/2}\,\frac{(q^{n+\gamma-\alpha};q){}_{\infty}}{(q^{n+\gamma-1};q){}_{\infty}}\, z^{1-n-\gamma}\,_{1}\phi_{1}(q^{\alpha-n-\gamma+1};q^{2-n-\gamma};q,-qz).\label{eq:q-confluent_sol_psi}
\end{eqnarray}
Here $z,\alpha,\gamma\in\mathbb{C}$, $q^{\gamma}\notin q^{\mathbb{Z}}$.
The recurrence rules imply that both $\left\{ \varphi_{n}\right\} $
and $\left\{ \psi_{n}\right\} $ solve the three-term difference equation
\begin{equation}
q^{\alpha+\gamma+n-1}\left(1-q^{\gamma-\alpha+n}\right)zu_{n+1}-\left(1-q^{\gamma+n-1}+q^{\gamma+n-1}z\right)u_{n}+u_{n-1}=0,\text{ }n\in\mathbb{Z}.\label{eq:q-confluent_diff_eq}
\end{equation}

\begin{lemma} The sequences $\left\{ \varphi_{n}\right\} $ and $\left\{ \psi_{n}\right\} $
defined in (\ref{eq:q-confluent_sol_phi}) and (\ref{eq:q-confluent_sol_psi}),
respectively, fulfill 
\begin{equation}
\varphi_{0}\psi_{1}-\varphi_{1}\psi_{0}=q^{-\alpha(\gamma+1)-\frac{1}{2}\gamma(\gamma-1)}(q^{\gamma-\alpha+1};q){}_{\infty}(-q^{\alpha}z;q){}_{\infty}\, z^{-\gamma}.\label{eq:q-confluent_wronskian}
\end{equation}
Alternatively, (\ref{eq:q-confluent_wronskian}) can be rewritten
as
\begin{eqnarray*}
 &  & \hskip-1em\,_{1}\phi_{1}(q^{\alpha};q^{\gamma};q,q^{\gamma-\alpha}z)\,_{1}\phi_{1}(q^{\alpha-\gamma};q^{1-\gamma};q,q^{1-\alpha}z)\\
\noalign{\smallskip} &  & \hskip-1em+\,\frac{q^{\gamma-1}\left(1-q^{\gamma-\alpha}\right)z}{\left(1-q^{\gamma-1}\right)\left(1-q^{\gamma}\right)}\,\,_{1}\phi_{1}(q^{\alpha};q^{\gamma+1};q,q^{\gamma-\alpha+1}z)\,_{1}\phi_{1}(q^{\alpha-\gamma+1};q^{2-\gamma};q,q^{1-\alpha}z)=(z;q)_{\infty}.
\end{eqnarray*}
\end{lemma}

%\begin{proof}[Proof of (\ref{eq:q-confluent_wronskian})]

\begin{proof} Checking the Wronskian of the solutions $\varphi_{n}$
and $\psi_{n}$ one finds that
\begin{equation}
q^{\frac{1}{2}n(n-1)+(\alpha+\gamma)n}\,(q^{\gamma-\alpha+1};q){}_{n}\, z^{n}(\varphi_{n}\psi_{n+1}-\varphi_{n+1}\psi_{n})=C\label{eq:aux_q-confluent_C}
\end{equation}
is a constant independent of $n$. In particular, $\varphi_{0}\psi_{1}-\varphi_{1}\psi_{0}=C$.
It is straightforward to examine the asymptotic behavior for large
$n$ of the solutions in question getting $\varphi_{n}=1+O(q^{n})$
and
\[
\psi_{n}=q^{-\frac{1}{2}n(n-1)-(\alpha+\gamma-1)n-\frac{1}{2}(\gamma-1)(\gamma-2)-\alpha\gamma}\, z^{1-\gamma-n}\,(-q^{\alpha}z;q){}_{\infty}\left(1+O(q^{n})\right).
\]
Sending $n$ to infinity in (\ref{eq:aux_q-confluent_C}) one finds
that $C$ equals the RHS of (\ref{eq:q-confluent_wronskian}). \end{proof}

\begin{proposition} For $\alpha,\gamma,z\in\mathbb{C}$,
\begin{equation}
\mathfrak{F}\!\left(\left\{ \frac{q^{\frac{1}{2}(\alpha+\gamma+k)-\frac{3}{4}}\,(q^{\gamma-\alpha+k};q^{2}){}_{\infty}\,\sqrt{z}}{(q^{\gamma-\alpha+k+1};q^{2}){}_{\infty}\left(1-(1-z)q^{\gamma+k-1}\right)}\right\} _{\! k=1}^{\!\infty}\right)\!=\frac{(q^{\gamma};q){}_{\infty}}{((1-z)q^{\gamma};q){}_{\infty}}\,\,_{1}\phi_{1}(q^{\alpha};q^{\gamma};q,-q^{\gamma}z).\label{eq:q-confluent_calF}
\end{equation}
\end{proposition}

%\begin{proof}[Proof of (\ref{eq:q-confluent_calF})]

\begin{proof} The both sides of the identity are regarded as meromorphic
functions in $z$. Setting $\Im\gamma$ to a constant, the both sides
tend to $1$ as $\Re\gamma$ tends to $+\infty$. In virtue of Lemma~\ref{lem:Fn_eq_calF},
it suffices to verify that the sequence
\[
F_{n}=\frac{(q^{\gamma+n-1};q){}_{\infty}}{((1-z)q^{\gamma+n-1};q){}_{\infty}}\,\,_{1}\phi_{1}(q^{\alpha};q^{\gamma+n-1};q,-q^{\gamma+n-1}z),\text{ }n\in\mathbb{N},
\]
satisfies the three-term recurrence relation $F_{n}-F_{n+1}+s_{n}zF_{n+2}=0$,
$n\in\mathbb{N}$, where
\[
s_{n}=\frac{q^{\alpha+\gamma+n-1}\left(1-q^{\gamma-\alpha+n}\right)}{\left(1-(1-z)q^{\gamma+n-1}\right)\left(1-(1-z)q^{\gamma+n}\right)}\,.
\]
Since $\gamma$ here is arbitrary one can consider just the equality
for $n=1$. But then the three-term recurrence coincides with (\ref{eq:q-confluent_recurr1})
(provided $z$ is replaced by $-z$). \end{proof}

Let us now focus on equation (\ref{eq:q-confluent_diff_eq}). One
can extract from it a solvable eigenvalue problem for a Jacobi matrix
obeying the convergence condition (\ref{eq:assum_sum_w}).

\begin{proposition}\label{prop:q-confluent_specJ} For $\sigma\in\mathbb{R}$
and $\gamma>-1$, let $J=J(\sigma,\gamma)$ be the Jacobi matrix operator
in $\ell^{2}(\mathbb{N})$ defined by (\ref{eq:matrixJ}) and
\begin{equation}
w_{n}=\frac{1}{2}\sinh(\sigma)q^{(n-\gamma-1)/2}\sqrt{1-q^{n+\gamma}}\,,\ \lambda_{n}=q^{n-1}.\label{eq:q-confluent_w_lbd}
\end{equation}
Then $z\neq0$ is an eigenvalue of $J(\sigma,\gamma)$ if and only
if
\[
\left(\cosh^{2}(\sigma/2)z^{-1};q\right){}_{\!\infty}\,\,_{1}\phi_{1}\!\left(q^{-\gamma}\cosh^{2}(\sigma/2)z^{-1};\cosh^{2}(\sigma/2)z^{-1};q,-\sinh^{2}(\sigma/2)z^{-1}\right)=0.
\]
Moreover, if $z\neq0$ solves this characteristic equation then the
sequence $\left\{ v_{n}\right\} _{n=1}^{\infty}$, with
\begin{eqnarray}
v_{n} & = & q^{-\frac{1}{2}\gamma n+\frac{1}{4}n(n-3)}\,\frac{\sinh^{n}(\sigma)\,(2z)^{-n}}{\sqrt{(q^{\gamma+n};q){}_{\infty}}}\left(q^{n}\cosh^{2}\left(\frac{\sigma}{2}\right)z^{-1};q\right){}_{\!\infty}\nonumber \\
\noalign{\smallskip} &  & \times\,\,_{1}\phi_{1}\!\left(q^{-\gamma}\cosh^{2}\!\left(\frac{\sigma}{2}\right)z^{-1};q^{n}\cosh^{2}\!\left(\frac{\sigma}{2}\right)z^{-1};q,-q^{n}\sinh^{2}\!\left(\frac{\sigma}{2}\right)z^{-1}\right),\label{eq:q-confluent_v_n}
\end{eqnarray}
is a corresponding eigenvector. \end{proposition}

\begin{remark} Notice that the matrix operator $J(\sigma,\gamma)$
is compact (even trace class). \end{remark}

\begin{proof} First, apply in (\ref{eq:q-confluent_diff_eq}) the
substitution
\[
\gamma=\tilde{\gamma}+\alpha,\ z=q^{\beta},\ u_{n}=q^{-\alpha n}\tilde{u}_{n},
\]
and then forget about the tilde over $\gamma$ and $u$. Next use
the substitution
\[
q^{\beta/2}=\tanh\!\left(\frac{\sigma}{2}\right)\!,\ q^{\alpha}=q^{-\gamma}\cosh^{2}\!\left(\frac{\sigma}{2}\right)\tilde{z}^{-1},\ u_{n}=\phi_{n}\tilde{u}_{n}.
\]
where $\left\{ \phi_{n}\right\} $ is a sequence obeying
\[
\frac{\phi_{n}}{\phi_{n+1}}=q^{(\beta+\gamma+n-1)/2}\,\sqrt{1-q^{\gamma+n}}\,.
\]
Up to a constant multiplier, $\phi_{n}{}^{2}=q^{-\beta n-\gamma n-\frac{1}{2}n(n-3)}\,(q^{\gamma+n};q){}_{\infty}$.
We again forget about the tildes over $z$ and $u$, and restrict
the values of the index $n$ to natural numbers. If $u_{0}=0$ then
the transformed sequence $\left\{ u_{k}\right\} _{k=1}^{\infty}$
solves the Jacobi eigenvalue problem (\ref{eq:eigenvalue}) with $w_{n}$
and $\lambda_{n}$ given in (\ref{eq:q-confluent_w_lbd}).

Further apply the same sequence of transformations to the solution
$\varphi_{n}$ in (\ref{eq:q-confluent_sol_phi}). Let us call the
resulting sequence $\left\{ v_{n}\right\} $. A straightforward computation
yields (\ref{eq:q-confluent_v_n}). Clearly, the sequence $\{v_{k};\ k\geq1\}$
is square summable. On general grounds, since $J(\sigma,\gamma)$
falls into the limit point case, any other linearly independent solution
of the recurrence in question, (\ref{eq:2ndorder_diff_eq_param_z}),
cannot be square summable. Hence the characteristic equation for this
eigenvalue problem reads $v_{0}=0$. This shows the proposition. \end{proof}

\begin{remark} In the particular case $\gamma=0$ the characteristic
equation simplifies to the form
\[
\left(\cosh^{2}(\sigma/2)z^{-1};q\right){}_{\!\infty}\left(-\sinh^{2}(\sigma/2)z^{-1};q\right){}_{\!\infty}=0.
\]
Hence in that case, apart of $z=0$, one knows the point spectrum
fully explicitly,
\[
\text{spec}J(\sigma,0)\setminus\{0\}=\left\{ q^{k}\cosh^{2}(\sigma/2);\, k=0,1,2,\ldots\right\} \cup\left\{ -q^{k}\sinh^{2}(\sigma/2);\, k=0,1,2,\ldots\right\} \!.
\]
\end{remark}

\begin{remark} Of course, Proposition~\ref{prop:q-confluent_specJ}
can be as well derived using formulas (\ref{eq:def_xi_k_meromorph}),
(\ref{eq:specJ_xi0}), while knowing that (\ref{eq:assum_sum_w})
is fulfilled. To evaluate $\xi_{n}(z)$ one can make use of (\ref{eq:q-confluent_calF}).
Applying the same series of substitutions as above to equation (\ref{eq:q-confluent_calF})
one gets
\begin{eqnarray*}
 &  & \mathfrak{F}\!\left(\left\{ \frac{q^{\frac{1}{2}(k-\gamma)-\frac{3}{4}}\,\sinh(\sigma)\left(q^{\gamma+k};q^{2}\right){}_{\!\infty}}{2\left(q^{\gamma+k+1};q^{2}\right){}_{\!\infty}\left(q^{k-1}-z\right)}\right\} _{\! k=1}^{\!\infty}\right)\\
\noalign{\smallskip} &  & =\,\frac{\left(\cosh^{2}(\sigma/2)z^{-1};q\right){}_{\!\infty}}{\left(z^{-1};q\right){}_{\!\infty}}\,\,_{1}\phi_{1}\!\left(q^{-\gamma}\cosh^{2}\!\left(\frac{\sigma}{2}\right)z^{-1};\cosh^{2}\!\left(\frac{\sigma}{2}\right)z^{-1};q,-\sinh^{2}\!\left(\frac{\sigma}{2}\right)z^{-1}\right)\!.
\end{eqnarray*}
Then a straightforward computation yields

\[
\xi_{n}(z)=\frac{2\, q^{(\gamma+1)/2}\,\sqrt{(q^{\gamma+1};q){}_{\infty}}}{\sinh(\sigma)\left(z^{-1};q\right){}_{\!\infty}}\, v_{n},\text{ }n=0,1,2,\ldots,
\]
with $v_{n}$ being given in (\ref{eq:q-confluent_v_n}). \end{remark}

\section*{Acknowledgments}

One of the authors (F.\v{S}.) wishes to acknowledge gratefully partial
support from grant No. SGS12/198/OHK4/3T/14 of the Grant Agency of
the Czech Technical University in Prague.

\end{document}